%% file: main.tex
\documentclass[10pt]{article}
\usepackage[utf8]{inputenc}
\usepackage{natbib}
\usepackage[hidelinks]{hyperref}
\usepackage[a4paper, total={7in, 10in}]{geometry}
\usepackage{graphicx}
\usepackage{array}
\usepackage{paralist}
\usepackage{amsmath}
\usepackage{amssymb}
\usepackage{amsthm}
\usepackage{xcolor}
\usepackage{mathtools}
\usepackage{float}
\newtheorem{theorem}{Theorem}[section]
\newtheorem*{acknowledgement*}{Acknowledgement}
\newtheorem{assumption}{Assumption}
\newtheorem{corollary}[theorem]{Corollary}
\newtheorem{definition}[theorem]{Definition}
\newtheorem{example}[theorem]{Example}
\newtheorem{lemma}[theorem]{Lemma}
\newtheorem{remark}[theorem]{Remark}
\def\mP{\mathbb P}
\DeclareMathAlphabet{\mathmybb}{U}{bbold}{m}{n}
\newcommand{\1}{\mathmybb{1}}

\DeclareMathOperator{\E}{\mathbb E}
\title{Splitting infinity: a de Finetti game with state-dependent profit rates and singular control for diffusions}
\author{
    Piotr Chlebicki\footnote{Corresponding author. E-mail: \url{piotr.chlebicki@math.su.se}}\\
    Department of Mathematics, Stockholm University\\
    \\Kristoffer Lindensjö\\Department of Mathematics, Stockholm University\\
}
\begin{document}
\maketitle
\begin{abstract}
    We study a game of resource extraction of a common good under one-dimensional diffusive dynamics with player actions corresponding to singular stochastic control up to absorption at $0$, implying a trade-off between profitable resource extraction and sustainability. Unsurprisingly, immediate extraction of all available resources is an equilibrium. A main result is that we characterize and prove the existence of non-trivial equilibria that do not result in immediate absorption, but instead are attained with both players extracting resources according to a state-dependent rate of threshold type, corresponding to the presence of control only when the state process is in an interval $(b,\infty)$. The underlying assumption is, roughly, that the drift coefficient of the uncontrolled state process grows sufficiently fast in relation to the discount rate, implying that the value for the corresponding one-player problem is infinite. We also study a generalization of the game that allows a state-dependent profit rate integrated against the control processes. In this game we again characterize and prove the existence of non-trivial equilibria of threshold type. In particular, a main novelty is that we find equilibria where the state process is controlled with its own local time such that we have reflection points with associated initial jumps, as well as other points in the state space where the control processes increase in a singular manner (skew points).
\end{abstract}
{\small	
} 
\input{introduction}
\input{model_and_assumptions}
\section{Equilibria with symmetric extraction rates above a threshold}
\label{sec:Naive_Lebesgue_absolutely_continuous_Nash_equilibria}
This section is devoted to finding global Markovian Nash equilibria of threshold type (for the game formalized in Section \ref{sec:Admissibility}, with a constant profit rate $g=1$) in line with the objectives outlined in Section \ref{sec:intro}.

We start by making some initial observations of motivational value, which lead up to the formal results Theorem \ref{Theorem-NE-better-verification-theorem} (equilibrium verification) and Theorem \ref{Theorem-NE-better-existence-theorem} (equilibrium existence). Our aim is to search for a symmetric global Markovian Nash equilibrium where each player uses an extraction rate $\lambda_{b}^{\ast}$  of threshold type, meaning that $\lambda_{b}^{\ast}$ only assumes positive values above some boundary $b\geq0$. Our ansatz is that the optimal (in the usual sense) singular stochastic control problem faced by each player when deciding whether to deviate from such a proposed equilibrium has a usual bang-bang type solution with reflection at $b$ and a downward jump to $b$ when starting the process above $b$ (note that we also need that another optimal solution is to use the rate $\lambda_{b}^{\ast}$, in order to have an equilibrium); i.e., our ansatz is to consider a candidate value function given by:
\begin{equation}\label{Vb-g=1}
    V_{b}(x)=\begin{cases}
     C_{1}\psi(x)+C_{2}\varphi(x) & \text{for } 0 \leq x<b\\
        x-b+C_{1}\psi(b)+C_{2}\varphi(b) & \text{for }x\geq b,
    \end{cases}
\end{equation}
where $\psi,\varphi\in\mathcal{C}^{2}\left[0,\infty\right)$ are the fundamental (unique) non-negative increasing and decreasing solutions to:
\begin{equation}\label{eq:ode-for-psi}
    A_{X} f(x) := \frac{\sigma^{2}(x)}{2}f''(x)+\mu(x)f'(x)=rf(x),
\end{equation}
with $\psi(0)=\varphi(\infty)=0$ and $\psi'(0)=\varphi(0)=1$; cf. e.g., \cite[p. 18-19]{borodin2012handbook}, \cite[Sec. 2]{ekstrom2023finetti} and \cite[Lemma 4.1]{shreve1984optimal}. 

Let us try to identify a corresponding candidate equilibrium strategy function $\lambda^\ast_b$, i.e., an equilibrium extraction rate which is such that if both players use it then the corresponding value function is of the form \eqref{Vb-g=1}. To this end we observe that the Hamilton-Jacobi-Bellman equation (HJB) for the optimization problem faced by a deviating player supposing that the other player uses a rate $\lambda^\ast_b$ (cf. e.g., \cite{ekstrom2022detect}), involves:
\begin{align}\label{HJB-first-thing}
\begin{split}
    \frac{\sigma^2(x)}{2} V_{b}''(x) + (\mu(x) - 2\lambda_{b}^{\ast}(x))V_{b}'(x) - 
    rV_{b}(x) +\lambda_b^{\ast}(x) & =0, \enskip\text{for: }x\neq b\\
    \frac{\sigma^2(x)}{2} V_{b}''(x) + (\mu(x) - \lambda_{b}^{\ast}(x)-\lambda)V_{b}'(x) -
    rV_{b}(x) +\lambda & \leq0, \enskip\text{for: }x\neq b \text{ and for all $\lambda\geq0$}.
\end{split}
\end{align}
Our ansatz is furthermore that $V_b'\geq1$, which implies that the second part of \eqref{HJB-first-thing} holds whenever the first part holds (cf. Remark \ref{remark:HJB-necessary}). 
Hence, when also observing that 
 $V_b'(x)=1, x > b$, we find a candidate equilibrium strategy of the form:
\begin{equation*}
    \lambda_b^\ast(x) = \1_{\{x > b\}}(A_X - r) V_b(x).
\end{equation*}
Using $V_b(0)=0$ and the ansatz that $V_b$ is continuously differentiable (implying $V_b'(b-)=1$) we find $C_1= \psi'(b)^{-1}$ and $C_2=0$, so that:
\begin{equation}\label{eq:Vb}
    V_{b}(x)=\begin{cases}
        \frac{\psi(x)}{\psi'(b)} & \text{for } 0 \leq x<b\\
        x - b + \frac{\psi(b)}{\psi'(b)} & \text{for }x\geq b.
    \end{cases}
\end{equation}
This  means that the candidate equilibrium extraction rate specifies to:
\begin{equation}\label{DE:lambda_b}
    \lambda_{b}^{\ast}(x) =\1_{\{x > b\}}
    \left(\mu(x)-rx-r\left(\frac{\psi(b)}{\psi'(b)}-b\right)\right).
\end{equation}

\begin{remark} 
    Let us consider the one-player problem that e.g., player $1$ faces in case player $2$ uses a particular threshold strategy $\lambda_b^\ast$. In order that the optimal value for this problem should be \eqref{eq:Vb}, one would perhaps guess that we should identify $b$ so that $V_b''(b-)=0$ according to the usual smooth fit condition of singular stochastic control; but this is not how we will proceed. Indeed, both players using the same threshold strategy $\lambda_b^\ast$ (given by \eqref{DE:lambda_b}) corresponds to a global Markovian Nash equilibrium for several different fixed $b$ (under suitable conditions, see Theorems \ref{Theorem-NE-better-verification-theorem} and \ref{Theorem-NE-better-existence-theorem}, below), and in general we do not have smooth fit. 
\end{remark}

\begin{theorem}[Verification]\label{Theorem-NE-better-verification-theorem}
    Suppose Assumptions \ref{Assumption_1} and \ref{Assumption_2} hold. Consider a constant $b\geq0$ and suppose that the following conditions hold:
        \begin{align}
        \label{NE:cond2}
        \text{$V_b$ defined in \eqref{eq:Vb} satisfies $V_b'(x)\geq 1,x\geq0$,}\\
        \text{$\lambda_b^\ast$ defined in \eqref{DE:lambda_b} satisfies 
        $\lambda_{b}^{\ast}(x)\geq0,x\geq0$.}\label{NE:cond1}
    \end{align}
    Then, the control strategies $D^i, i=1,2$ given by:
    \begin{align*}
        D^{i}=\{\lambda_b^\ast,\emptyset,\emptyset,0,\emptyset \},
    \end{align*}
    (corresponding to each player $i=1,2$ controlling only with the rate $\lambda_b^\ast$) constitute a global Markovian Nash equilibrium with an  equilibrium value given by $V_b$.
\end{theorem}
\begin{proof}
    This result is a specific case of Theorem \ref{Theorem_verification_g_NE_2} which is proved in Section \ref{subsec:proof_g_verification} below.
\end{proof}
In relation to \eqref{NE:cond2} we note that $V_b'(x)=1$  for $x\geq b$, for any $b\geq0$.

\begin{remark}\label{remark:HJB-necessary}
    Note that condition \eqref{NE:cond2} is necessary and sufficient for the inequality:
    \begin{equation*}
        \lambda(1-V_{b}'(x))= 
        \frac{1}{2}\sigma^2(x) V_{b}''(x) + (\mu(x) - \lambda_{b}^\ast(x) - \lambda)V_{b}'(x) - 
        rV_{b}(x) + \lambda \leq0,
    \end{equation*}
    to hold for all $\lambda\geq0$ and $x\neq b$ (the equality in the above is due to \eqref{eq:Vb} and \eqref{DE:lambda_b}). This inequality corresponds to the inequality in \eqref{HJB-first-thing}. The interpretation of \eqref{NE:cond1} is that this condition implies that the strategy $\{\lambda_b^\ast,\emptyset,\emptyset,0,\emptyset \}$ is admissible. Therefore an intuitive interpretation of assumptions \eqref{NE:cond2} and \eqref{NE:cond1} is that together they correspond to $\lambda_b^\ast$ being the maximizing rate of the HJB \eqref{HJB-first-thing} for the optimization problem of one of the players, supposing that the other player uses the same rate.
\end{remark}

\begin{theorem}[Equilibrium existence]\label{Theorem-NE-better-existence-theorem}
    Suppose Assumptions \ref{Assumption_1} and \ref{Assumption_2} hold. 
    \begin{enumerate}[(I)]
        \item There exists a constant $\underline{b}\geq 0$ such that conditions \eqref{NE:cond2} and \eqref{NE:cond1} hold (and we thus have a global Markovian Nash equilibrium as in Theorem \ref{Theorem-NE-better-verification-theorem}), for any $b \geq \underline{b}$. 
        \item Consider a fixed for $b\geq 0$. If $x\mapsto\mu(x)-rx$ is non-decreasing on $[0, b)$ with $\mu(0)\geq0$, then \eqref{NE:cond2} holds. Hence, if also \eqref{NE:cond1} holds, then we have a global Markovian Nash equilibrium as in Theorem \ref{Theorem-NE-better-verification-theorem} for $b$.
        \item If $x\mapsto\mu(x)-rx$ is non-decreasing on $[0, \infty)$ with $\mu(0)\geq0$ then  \eqref{NE:cond2} and \eqref{NE:cond1} hold (and thus we have a global Markovian Nash equilibrium as in Theorem \ref{Theorem-NE-better-verification-theorem}), for any $b\geq 0$.
    \end{enumerate}
\end{theorem}
\begin{proof}
    This result is a specific case of Theorem \ref{Theorem_existence_g_NE_1} proved in Section \ref{subsec:proof_g_existence} below.
\end{proof}

\begin{corollary}\label{cor:value_goes_to_inf} 
    Suppose Assumptions \ref{Assumption_1} and \ref{Assumption_2} hold. Then, there exists a constant $\underline{b} \geq 0$ such that we have a global Markovian Nash equilibrium for any $b \geq \underline{b}$, and  the equilibrium value $V_b(x)$ is strictly increasing in $b$ on $(\underline{b},\infty)$ for fixed $x>0$, with:
    \begin{equation*}
        \lim_{b\rightarrow\infty}V_b(x)=\infty.
    \end{equation*}
\end{corollary}

\begin{proof}
    To see that the first claim holds, we select $\underline{b}$ so that the conditions of Theorem \ref{Theorem-NE-better-existence-theorem}(I) holds. To see that the second claim holds we further specify $\underline{b}$ so that we also have $\underline{b}\geq b_2$ (using the notation for $b_2$ from Lemma \ref{Lemma_that_I_don't_know_what_to_call}(II)), and then use \eqref{eq:Vb} together with Lemma \ref{Lemma_that_I_don't_know_what_to_call}(II). The third claim follows similarly from Lemma \ref{Lemma_that_I_don't_know_what_to_call}(III).
\end{proof}

\begin{example}\label{first-example}
    Let $\mu(x)=\mu x$ and $\sigma(x)=\sigma$, where $\mu>r$ and $\sigma>0$ are here constants, so that the uncontrolled process is an Ornstein-Uhlenbeck process without mean-reversion. Assumptions \ref{Assumption_1} and \ref{Assumption_2} are immediately verified. It is directly verified that the conditions of Theorem \ref{Theorem-NE-better-existence-theorem}(III) holds, so that we have a global Markovian Nash equilibrium of the kind in Theorem \ref{Theorem-NE-better-existence-theorem}, for any threshold $b \geq 0$. The value function $V_b$ and extraction rate $\lambda_b^\ast$ for these equilibria can be calculated using standard numerical methods for ordinary differential equations (ODEs).  Illustrations are provided in Figure \ref{fig:V_lambda_x_b}.
    \begin{figure}[H]
        \centering
        \includegraphics[width=0.45\linewidth]{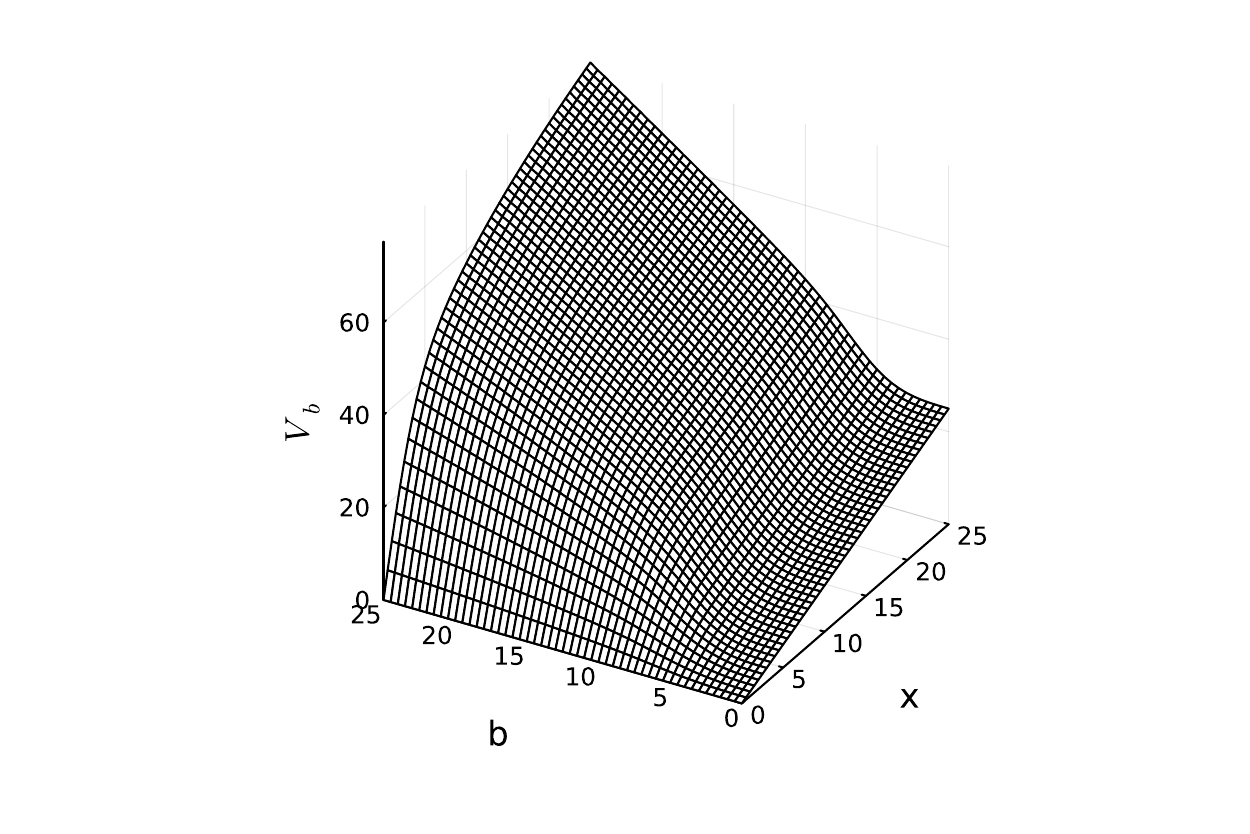}
        \includegraphics[width=0.45\linewidth]{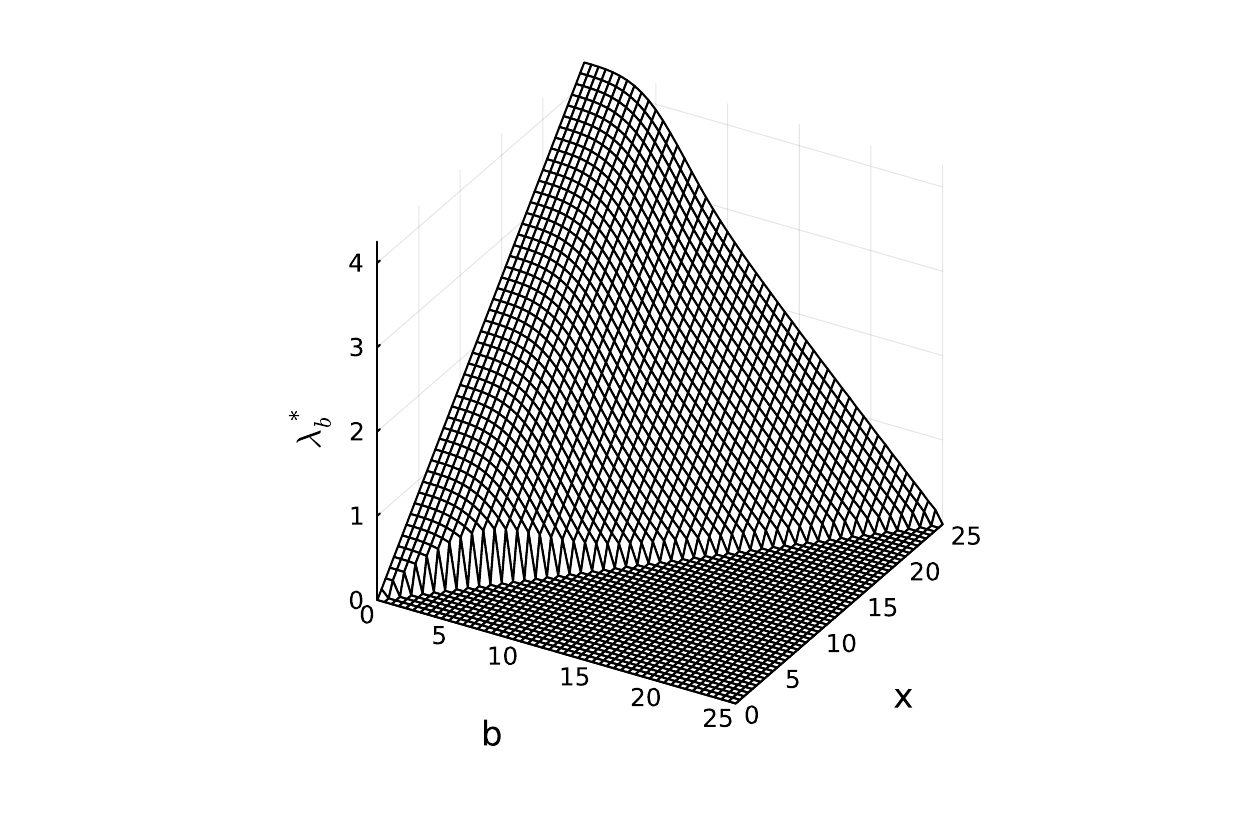}
        \caption{Illustrations for Example \ref{first-example}, for  for the parameters $r=0.08,\mu=0.25$ and $\sigma=2$. The value function $V_b$ is increasing in $x$ as well as in $b$. The extraction rate $\lambda_b^\ast$ is increasing in $x$ and decreasing in $b$.}
        \label{fig:V_lambda_x_b}
    \end{figure}
\end{example}

\input{g_function.tex}
\input{appendix}
\subsection*{Acknowledgements}
The authors would like to extend gratitude to Andi Bodnariu (Stockholm University) and Erik Ekström (Uppsala University) for fruitful discussions regarding the topic of the present paper. 

\bibliographystyle{abbrv}
\bibliography{Bibl}
\end{document}

%% file: introduction.tex
\section{Introduction}\label{sec:intro}
Consider a complete filtered probability space $(\Omega, {\cal F}, \mP, ({\cal F}_t)_{t \geq 0})$ that satisfies the usual conditions and supports a standard one-dimensional Wiener process $(W_t)_{t\geq0}$. Consider two players $i=1,2$ with the ability to control a one-dimensional diffusion process $X=(X_{t})_{t\geq0}$, corresponding to the level of available resources, which in the absence of control is the unique strong solution to the stochastic differential equation (SDE):
\begin{align}\label{SDE_uncontrolled}
    dX_t = \mu\left(X_t\right)dt + \sigma\left(X_t\right)dW_t, \enskip X_{0} = x \geq 0,  \enskip
    0 \leq t \leq \tau:=\inf\{t\geq0:X_{t}\leq0\}.
\end{align}
In particular, the state process is absorbed when reaching $0$. The dynamics of the controlled state process $X^{D^1, D^2}=\left(X_{t}^{D^1,D^2}\right)_{t\geq0}$ are determined by two non-decreasing right-continuous with left limits (RCLL) processes $D^i=\left(D_{t}^{i}\right)_{t\geq0}$, corresponding to the accumulated extraction level of each of the players, according to the SDE:
\begin{align}\label{SDE_controlled_introduction}
    dX^{D^1,D^2}_t = \mu\left(X^{D^1,D^2}_t\right)dt + 
    \sigma\left(X^{D^1,D^2}_t\right)dW_t - dD_{t}^{1} - dD_{t}^{2}, \enskip 
    0 \leq t \leq \tau^{D^1,D^2}:=\inf\left\{t\geq0:X^{D^1,D^2}_{t}\leq0\right\}.
\end{align}
To ease the notation we will, throughout the paper, write $\tau$ instead of e.g., $\tau^{D^1,D^2}$. In our game the players do not select the processes $D^i$  directly, instead we allow them to select control strategies from a certain set of generalized Markov controls allowing classical control rates, jumps, and non-singular drift control in terms of local time (with respect to the controlled process) pushes with state-dependent intensities (skew points), where an aggregated, among the players, intensity of $1$ implies reflection; see Definition \ref{def:Markovian-admissible-controls} for details. 

In case at most one player selects a control strategy with jumps, the rewards of the players are defined in line with the usual singular stochastic control formulation of the de Finetti problem:
\begin{equation}\label{eq_value-func-no-jump}
    J_i\left(x,D^1,D^2\right)  := \E_x\left[\int_{[0,\tau]} e^{-rt}dD_{t}^{i}\right].
\end{equation}
We remark that the set of admissible control strategies will only allow jumps such that $\Delta D^{i}_{t}:=D^{i}_{t}-D^{i}_{t-}\leq X_{t-}^{D^1,D^2}$. However, in order to account for the possibility of both players creating jumps in the state process, in particular at the same time, in a way so that $\tau$ (absorption) occurs with $X^{D^1,D^2}_{\tau-} < \Delta D_{\tau}^1+\Delta D_{\tau}^2$, meaning that the players collectively try to extract more resources than are available, we modify \eqref{eq_value-func-no-jump} by defining the expected rewards as:
\begin{equation}\label{eq_value-func}
    J_i\left(x,D^1,D^2\right) := 
    \E_{x}\left[\int_{0}^\tau e^{-rt}d\left(D_{t}^{i}\right)^{c}+
    \sum_{0\leq t\leq\tau}e^{-rt}\frac{\Delta D_{t}^{i}}{\Delta D_{t}^{1} + \Delta D_{t}^{2}}
    \left(\left(\Delta D_{t}^{1} + \Delta D_{t}^{2}\right)\land X_{t-}^{D^1, D^2}\right)\right].
\end{equation} 
Note, $\left(D^i\right)^c$ denotes the continuous part of $D^i$, $r>0$ is a constant discount rate, $\E_x[\cdot]:= \E\left[\cdot\left| X^{D^1,D^2}_{0-}=x\right.\right]$, and the integrals are Lebesgue-Stieltjes. We also use the convention that the sum in \eqref{eq_value-func} should be interpreted to be over the jump times of $X^{D^1, D^2}$, i.e., when $\Delta D_{t}^{1} + \Delta D_{t}^{2}>0$. We refer to Section \ref{sec:Admissibility} for further details. 

Note that the second part of \eqref{eq_value-func} corresponds to payouts due to jumps in the control processes, and also that \eqref{eq_value-func} collapses into \eqref{eq_value-func-no-jump} in case there are jumps in at most one of the control processes $D^i$. Moreover, note that if the players collectively try, at some time $t-$, to extract more resources than are currently available (i.e., $\Delta D_{t}^{1}+\Delta D_{t}^{1} \geq X_{t-}^{D^1,D^2}$, so that absorption occurs immediately, i.e., $t=\tau$) then the total payout is capped by the amount of available resource (i.e., $X_{t-}^{D^1,D^2}$) and the individual payouts are proportional to the total attempted extraction amount. In particular, this means that if both players try to extract all available resources immediately, i.e., $\Delta D_{0}^{1}=\Delta D_{0}^{2}=X_{0-}=x$, then absorption occurs immediately and the players share the available resources equally, i.e., $J_i\left(x,D^1,D^2\right)=\frac{x}{2}$.

There is a vast literature on the problem of recourse extraction in a random environment, which traces its roots to the problem of forest rotation (the Faustmann problem); for modern studies of this problem see  e.g., \cite{alvarez2004stochastic,alvarez2007optimal,willassen1998stochastic}, and for a recent literature survey in view of the current problem formulation see \cite[Sec. 1]{ekstrom2023finetti}, where additionally the interpretation of the de Finetti problem as a problem of recourse extraction is further clarified. 

A main aspect of the literature on the de Finetti problem has been to identify conditions under which the optimal solution is of threshold type, in the sense that there is no control of the state process below some threshold $b \in [0,\infty)$ and maximal extraction above $b$. One common set of model assumptions ensuring optimal solutions of threshold type in the de Finetti problem with an underlying one-dimensional diffusion is that the drift and the diffusion coefficients satisfy $\mu'<r$ (i.e., the rate of growth of the drift coefficient is smaller than the discount rate) with $\mu(0)>0$ and $\sigma>0$  (or generalizations thereof), combined with some regularity conditions for $\mu$ and $\sigma$. Among the large variety of works studying this subject we mention only a small selection; see the seminal  \cite{shreve1984optimal} as well as e.g.,  \cite{asmussen1997controlled, de2017dividend, locas2024finetti, lokka2008optimal, paulsen2003optimal, paulsen2008optimal}. For surveys of the literature on the de Finetti problem interpreted as a problem of optimal dividends we refer to \cite{albrecher2009optimality,avanzi2009strategies}; for more recent literature surveys see e.g., \cite[Sec. 1.3]{bandini2022optimal} and \cite[Sec. 1]{lindensjo2020optimal}.  

In \cite{ekstrom2023finetti} a game version of the de Finetti problem was studied in a setting allowing for Markov controls corresponding to classical rate control for a one-dimensional diffusion. Under the assumption that the admissible control rates are bounded by a constant it was therein established that there is for this game a Nash equilibrium solution of threshold type under conditions similar to those mentioned in the previous paragraph. However, it is easily seen that if we consider more general control strategies allowing for jumps in the setting of \cite{ekstrom2023finetti}, then the equilibrium is trivial in the sense that the players will extract all resources immediately (due to pre-emption). 

The main aim of the present paper is to search for Nash equilibria, for the de Finetti problem, of threshold type in the sense that there is in equilibrium no control below a threshold $b$ while there is non-negative control above $b$ (although not necessarily maximal), in a setting that admits more general Markov control strategies, allowing in particular jumps. To attain this aim our main assumption is, contrary to the standard assumption mentioned above, that the rate of growth of the drift coefficient $\mu$ is larger than the discount rate $r$ for large values of the state process; see Section \ref{subsec:assum} for details. An interpretation of this assumption is that a threshold equilibrium can be attained in case the value of the corresponding one-player problem is infinite; indeed it can be shown that our assumptions yield an infinite value for the corresponding one-player problem (see Remark \ref{infinite-vaue-1-player-probem} for details).

More broadly, the aim of the present paper is to study a generalization of the de Finetti game presented  above by allowing a class of state-dependent profit rates corresponding to RCLL functions $g:[0,\infty) \rightarrow (0,\infty)$. In particular, the generalized game is based on the expected rewards:
\begin{equation}\label{eq:-circ-in-intro}
    \E_{x}\left[\int_{[0,\tau]} e^{-rt}g\left(X^{D^1,D^2}_{t-}\right)\circ d D_{t}^{i}\right],
\end{equation}
where our definition of the integral is stated in Section \ref{sec:g-function}. (We remark that the special case $g=1$ corresponds to the game presented above, with expected reward \eqref{eq_value-func}, and that it is considered separately from the general game formulation for ease of exposition.) 

An early work studying a singular stochastic control problem with a state-dependent profit rate  (also called reward rate) is \cite{zhu1992generalized}, which studies a (one-player) optimization problem of singular stochastic control for a diffusion and defines an integral for \eqref{eq:-circ-in-intro} that is used under the assumption that $g$ (using our notation) is continuos. This type of integral has recently been studied in several other works for the case that $g$ is continuos; see e.g., \cite{de2018stochastic, kwon2015game, liang2024equilibria} (in the context of singular stochastic control games) and \cite{ferrari2019class,jack2008singular,junca2024optimal, lon2011model} (in the context of one-player singular stochastic control problems). 

From a theoretical view-point the present paper contributes to the literature on nonzero-sum games of singular stochastic control, which is a class of problems that is far from being fully understood. Let us now mention some related works within this literature; for further related literature see e.g., \cite[Sec 1.4]{de2023nash} and \cite[Sec 1.1]{de2024model}. 

The work \cite{kwon2015game} studies a two-player game of singular stochastic control and stopping for a one-dimensional diffusion. The attention is limited to a class of barrier strategies that separate the state space into regions with (maximal) control and regions with no control, corresponding to controlling with a combination jumps and reflection of the state process.    
The work \cite{de2018stochastic} establishes a connection between a two-player nonzero-sum game of singular stochastic control and an associated game of stopping for a one-dimensional diffusion. Attention is restricted to searching for equilibria of double threshold type, which for the control game means reflection of the state processes at an upper and a lower threshold. 
The work \cite{de2023nash} studies an optimal dividend game allowing for singular control of Brownian motion with drift. A main feature is that each player controls their own state process. An important contribution is a definition of a general class of randomized feedback strategies for singular stochastic control. 
The work \cite{steg2024strategic} studies a singular stochastic control game of strategic irreversible investment where the (non-decreasing) control processes do not affect the state process; instead they determine the investment level of each player. The admissible controls correspond to a certain class of closed-loop strategies that allows for reflective behaviour. The work \cite{cao2023stationary} studies a stationary discounted and ergodic mean field game of singular control for a one-dimensional diffusion. 

The works that are closest to the present paper that we are aware of are \cite{de2018stochastic,ekstrom2023finetti,kwon2015game} (described above). The main differences of the present paper compared to \cite{ekstrom2023finetti}  are that the latter considers only the special case of a constant profit rate ($g=1$, in our notation), and that only Markov control strategies corresponding to bounded classical control rates are admissible. Major differences of the present paper compared to \cite{de2018stochastic} as well as \cite{kwon2015game} are that these works search for equilibria only in strategies of bang-bang type (see the above for details) and that the considered profit rates are continuos.    

Let us summarize the main findings and contributions of the present paper: 

We define a generalized set of Markov control strategies allowing classical control rates as well as jumps and skew points which are allowed to be simultaneous (Section \ref{sec:Admissibility}). To our knowledge this corresponds to a new type of Markov control strategies for singular stochastic control games; cf. e.g.,  \cite{de2018stochastic,kwon2015game} where threshold strategies corresponding to reflection and associated jumps are considered, \cite{aid2020nonzero,christensen2022moment,ferrari2019strategic} where impulse control strategies are considered, \cite{bodnariu2025time} where a novel class of control strategies is defined, allowing for jumps and reflective boundaries as well as exploding control rates that create inaccessible boundaries for the controlled process, and e.g., \cite{mannucci2004nonzero} where classical rate control strategies are considered; see also \cite[Sec. 2]{possamai2020zero} for a related discussion of different formulations of stochastic differential games. A related main finding of the present paper, that shows the utility of our generalized class of Markov control strategies, is that we find equilibria in which the control processes involve skew points and classical rate control (Section \ref{sec:g-function}) as well as the usual (in the context of singular stochastic control theory) bang-bang type control of reflecting at a given boundary and jumping to that boundary if the state process starts above it (Section \ref{sec:solution-with-jumps}). This marks, according to our knowledge, the first example of a control problems where skew points appear naturally in the solution. 

A main contribution for the game version of the usual de Finetti problem, i.e., with a constant profit rate (cf. the expected reward \eqref{eq_value-func}), is that we consider a setting in which immediate extraction of all resources is a admissible---indeed, we show that such immediate extraction corresponds to a trivial Nash equilibrium. In particular, the main contribution is that we also establish, under certain assumptions, existence and characterization results also for non-trivial Nash equilibria (Section \ref{sec:Naive_Lebesgue_absolutely_continuous_Nash_equilibria}).

We also provide corresponding equilibrium existence and characterization results for the generalized version of the de Finetti game allowing for a state-dependent profit rate integrated against the control processes, with expected reward \eqref{eq:-circ-in-intro}. Here we also consider for the first time, according to our knowledge, a discontinuous profit rate $g$ for the integral  \eqref{eq:-circ-in-intro}. We also allow simultaneous jumps and simultaneous singular control in terms of skew points, thereby contributing the question of how to to handle such simultaneous control (which has not been fully settled, cf. e.g., \cite[Appendix A.1]{de2018stochastic}, \cite[Sec. 2.2]{kwon2015game}, and \cite[Sec. 2]{aid2020nonzero}). We also consider absorption of the state process (at zero). In particular, these features motivate us to introduce a generalized definition for the integral in \eqref{eq:-circ-in-intro} (Section \ref{sec:g-function}). 

We believe that our framework and theoretical results can be useful for finding equilibria also in many other types of games involving singular stochastic control. 
 
The remainder of the paper is structured as follows. In Section \ref{sec:Model-and-assumptions} we formalize the setting of our games and establish a trivial Nash equilibrium corresponding to immediate absorption for the expected reward \eqref{eq_value-func}.  In Section \ref{sec:Naive_Lebesgue_absolutely_continuous_Nash_equilibria}, we report verification and existence results for threshold equilibria for the game corresponding to the expected reward \eqref{eq_value-func} and present an illustrative example. In Section \ref{sec:g-function}, we formalize the generalized game with a state dependent profit rate, establish a trivial Nash equilibrium, report verification and existence results for threshold equilibria, and present illustrative examples. In Section \ref{sec:solution-with-jumps}, we present a case study of our game with a state-dependent profit rate with an equilibrium strategy of bang-bang control type corresponding to reflection at a given boundary and a jump to that boundary if the state process starts above. 

%% file: model_and_assumptions.tex
\section{Problem formulation and initial observations}\label{sec:Model-and-assumptions}
In Section \ref{sec:Admissibility} we give the definitions of the set of admissible control strategies as well as the Nash equilibrium. In Section \ref{NE:Bad-subsection} we establish the existence of a trivial Nash equilibrium corresponding to immediate absorption. In Section \ref{subsec:assum} we present the assumptions that are used in the remainder of paper. 

\subsection{Admissible control strategies and global Markovian Nash equilibrium}\label{sec:Admissibility}
Let us first define our notion of admissible control strategies and after this explain how these correspond to controlled SDEs. In view of the Markovian formulation of our game, we consider Markov control strategies defined as follows. 
\begin{definition}[Admissible control strategies]\label{def:Markovian-admissible-controls}
    The set of admissible control strategies is denoted by $\mathbb{L}$. An admissible control strategy $D^i,i=1,2$ corresponds to a 5-tuple:
    \begin{equation*}
        D^i = \left\{\lambda^{i}(\cdot), \{x^{i}_j:j=1,\ldots,n_{i}\},
        \{c^{i}_j:j=1,\ldots,n_{i}\}, \mathcal{J}^{i}(\cdot),B^{i}\right\},
    \end{equation*}
    where:
    \begin{itemize}
        \item $\lambda^{i} :[0,\infty)\rightarrow[0,\infty)$ is a measurable function,
        \item $B^{i}\subseteq [0,\infty)$ is a finite union of closed and connected subsets of $[0,\infty)$,
        \item $\mathcal{J}^{i}:[0,\infty)\rightarrow[0,\infty)$ is a measurable function satisfying: $\mathcal{J}^{i}(x) = 0$ for $x\notin B^{i}$, $x - \mathcal{J}^{i}(x)\geq0$ for all $x\geq0$, and $\mathcal{J}^{i}(x)>0,x \in\mathrm{Int}\left(B^{i}\right)$, 
        \item $x^{i}_j \in  (0,\infty) \setminus\mathrm{Int}\left(B^{i}\right)$ and $c^{i}_j\in (0,1]$ are constants, and $n_{i} \in \mathbb{N}$.
    \end{itemize}
\end{definition}
For a given pair of admissible control strategies $\left(D^1,D^2\right)\in\mathbb{L}\times\mathbb{L}$ the controlled SDE is defined as:
\begin{align}\label{SDE}
\begin{split}
    dX^{D^1,D^2}_t & = \left(\mu-\sum_{i=1}^{2}\lambda^{i}\right)\left(X^{D^1,D^2}_t\right)dt + 
    \sigma\left(X^{D^1,D^2}_t\right)dW_t - 
    \sum_{i=1}^{2}\sum_{j=1}^{n_{i}} c^{i}_j dL_t^{x^{i}_j}\left(X^{D^1,D^2}\right)\\
    &- \sum_{i=1}^{2}\sum_{j = 0}^{\infty}\mathcal{J}^{i}
    \left(\alpha^{j}\left(X_{t-}^{D^1, D^2}\right)\right), \enskip X^{D^1,D^2}_{0-}= x, \enskip
    0 \leq t \leq \tau=\inf\left\{t\geq0:X^{D^1,D^2}_{t}\leq0\right\},
\end{split}
\end{align}
where $\alpha^0(x):=x, \alpha(x):= \left(x - \sum_{i=1}^{2}\mathcal{J}^{i}(x)\right)\lor0$ and $\alpha^j$ denotes $j$ compositions of $\alpha$, and $\left(L_t^{y}\left(X^{D^1,D^2}\right)\right)_{t\geq0}$ is the symmetric local time of $X^{D^1,D^2}$ (see e.g., \cite[Sec. 2.3]{Peskir}).

When there is no risk of confusion, we will write $X$ instead of $X^{D^1,D^2}$.  With a slight abuse of notation we also use the notation $\left(D^i_t\right)_{t\geq0}$, corresponding to to the control strategy $D^i$, to refer to the control processes in \eqref{SDE} given by:
\begin{equation}\label{SDE-control-processes}
    dD^i_{t} =
    \lambda^{i}\left(X^{D^1, D^2}_{t}\right)dt + 
    \sum_{j=1}^{n_{i}}c^{i}_{j} dL^{x_{j}^{i}}_{t}\left(X^{D^1, D^2}\right) + 
    \sum_{j = 0}^{\infty}\mathcal{J}^{i}\left(\alpha^{j}\left(X_{t-}^{D^1, D^2}\right)\right),
    \enskip D^i_{0-}=0,\enskip0\leq t\leq\tau,
\end{equation}
in case a unique strong solution $X^{D^1, D^2}$ to \eqref{SDE} exists; the case of non-existence is treated in Definition \ref{def:value-funcs}.

\begin{remark}[Interpretation of admissible control strategies and the corresponding controlled SDE \eqref{SDE}] \enskip
\begin{itemize}
    \item The function $\lambda^i$ corresponds to the part of the control of player $i$ that is absolutely continuous with respect to the Lebesgue measure. 
    \item The constants $x_j^i$ and $c_j^i$ correspond to player $i$ controlling with a local time push with intensity $c_j^i$ when $X_{t}^{D^1,D^2}=x_j^i$.
    \item  The interpretation of $B^i$ and $\mathcal{J}^{i}$ is that player $i$ controls with a jump of size $\mathcal{J}^{i}(X_{t-}^{D^1,D^2})$ when $X_{t-}^{D^1,D^2}\in B^{i}$. The conditions for $\mathcal{J}^{i}$ imply that $\Delta D^{i}_{t}\leq X_{t-}^{D^1,D^2}$. 
    \item The interpretation of $\alpha(X_{t-}^{D^1,D^2})=
    \left(X_{t-}^{D^1,D^2} - \mathcal{J}^{1}(X_{t-}^{D^1,D^2}) - \mathcal{J}^{2}(X_{t-}^{D^1,D^2})\right)\vee 0$ is that it is the value of the state process after the aggregated jump $\mathcal{J}^{1}(X_{t-}^{D^1,D^2})+\mathcal{J}^{2}(X_{t-}^{D^1,D^2})$. Note that it may be case that this state $\alpha(X_{t-}^{D^1,D^2})$ is also in the jump set, i.e., in $B^{1} \cup B^{2}$, which reveals the purpose of $\alpha$ in \eqref{SDE}; that is, $\alpha$ formalizes how the state process is affected by jumps that occur at the same time. We remark that the situation simplifies if only one player jumps, say $B^{2} = \emptyset$, and the other player jumps only once at a time in the sense that $x\mapsto \mathcal{J}^{1}\left(x-\mathcal{J}^{1}(x)\right)=0$. In this case \eqref{SDE} can be expressed without $\alpha$ as:
    \begin{align*}
        dX^{D^1,D^2}_t & = 
        \left(\mu-\sum_{i=1}^2\lambda^{i}\right)\left(X^{D^1,D^2}_t\right)dt + 
        \sigma\left(X^{D^1,D^2}_t\right)dW_t - 
        \sum_{i=1}^2\sum_{j=1}^{n_{i}} c^{i}_j dL_t^{x^{i}_j}\left(X^{D^1,D^2}\right)\\
        &- \mathcal{J}^{1}\left(X_{t-}^{D^1, D^2}\right),
        \enskip X^{D^1,D^2}_{0-}= x, \enskip  0 \leq t \leq \tau.
    \end{align*}
\end{itemize}
\end{remark}
With a clear understanding of admissible controls, we are now ready to give a formal definition of our expected rewards \eqref{eq_value-func}.
\begin{definition}[Expected rewards]\label{def:value-funcs} 
    Consider a pair of admissible control strategy $\left(D^1,D^2\right)\in \mathbb{L}\times\mathbb{L}$. Suppose the corresponding controlled SDE \eqref{SDE} has a unique strong solution for a fixed $X_{0-}=x \geq0$, then the expected rewards $J_i\left(x,D^1,D^2\right)$ are given by \eqref{eq_value-func} where 
    $\left(D^i_t\right)_{t\geq0}$ are given by \eqref{SDE-control-processes}; and if no unique strong solution exists then:
    \begin{equation*}
        J_i\left(x; D^1, D^2\right):=-\infty. 
    \end{equation*}
\end{definition}
We remark that the convention of assigning the value minus infinity in case the controlled process does not have a solution has been considered before; cf. e.g., \cite{de2023nash,possamai2020zero}.
\begin{remark}
    There is a large literature on existence and uniqueness of SDEs involving local times; see e.g., \cite{bass2007pathwise,bass2005one,blei2013one,blei2013note,le2006one,Lejay2006}. Whenever needed we will refer to relevant existence results in our proofs.
\end{remark}
We are now ready to present the equilibrium definition that we will study in the present paper. We remark that our notion of equilibrium corresponds to so called Markov perfect equilibria.
\begin{definition}[Global Markovian Nash equilibrium]\label{def:Markov-NE}
    A pair of admissible control strategies $(D^1,D^2)\in \mathbb{L}\times\mathbb{L}$ with corresponding controlled SDE \eqref{SDE} is said to be a global Markovian Nash equilibrium if:
    \begin{align*}
        J_1\left(x,D^{1},D^{2}\right) & \geq J_1\left(x,D,D^{2}\right)\\
        J_2\left(x,D^{1},D^{2}\right) & \geq J_2\left(x,D^{1},D\right),
    \end{align*}
    for all admissible control strategies $D\in\mathbb{L}$, and all $x \geq 0$. 
    If $\left(D^1,D^2\right)$ is a global Markovian Nash equilibrium then we say that $J_i\left(x,D^{1},D^{2}\right)$  are the equilibrium values; and if these values coincide then we denote the common equilibrium value by $V(x)$, i.e., $V(x)=J_i\left(x,D^{1},D^{2}\right)$.
\end{definition}

\subsection{Trivial equilibria}\label{NE:Bad-subsection}
It seems intuitive that both players immediately trying to extract more resources than are available (resulting in immediate absorption) should correspond to a Nash equilibrium, and this is indeed true. 

\begin{theorem}\label{thm:trivialNE} 
    A pair of admissible control strategies $(D^1, D^2)\in \mathbb{L}\times\mathbb{L}$ satisfying $B^{i} = [0, \infty)$ and ${\cal J}^{i}(x)=x$ (which implies that $D^i_0= x$ and $\tau=0$ a.s., for any $X_{0-}=x$), is a global Markovian Nash equilibrium. The corresponding equilibrium values are given by:
    \begin{equation*}
        V(x)= \frac{x}{2}, \enskip x \geq 0.
    \end{equation*}
\end{theorem}
\begin{proof}
    This follows immediately from our definitions; see \eqref{eq_value-func} and Definitions \ref{def:value-funcs} and \ref{def:Markov-NE}. 
\end{proof}

\subsection{Assumptions and the one-player problem in our setting}\label{subsec:assum}
The results of Sections \ref{sec:Naive_Lebesgue_absolutely_continuous_Nash_equilibria} and \ref{sec:g-function} rely on Assumptions \ref{Assumption_1} and \ref{Assumption_2}. Section \ref{sec:solution-with-jumps} relies only on Assumption \ref{Assumption_1}.
\begin{assumption}\label{Assumption_1}
    $\mu:[0,\infty) \rightarrow \mathbb{R}$ and $\sigma:[0,\infty) \rightarrow (0,\infty)$
    are globally Lipschitz continuous.
\end{assumption}
\begin{assumption}\label{Assumption_2} 
    There exists constants $\kappa \geq 0$ and $c>0$ such that $x\mapsto \mu(x) - (r + c)x$ is strictly increasing on $[\kappa, \infty)$.
\end{assumption}
\begin{remark}[On the one-player problem in our setting] \label{infinite-vaue-1-player-probem} 
    It can be shown that Assumptions \ref{Assumption_1} and \ref{Assumption_2} imply that the one-player problem (i.e., of maximizing \eqref{eq_value-func-no-jump}, or equivalently \eqref{eq_value-func}, in the usual sense),  gives an arbitrarily large expected reward for each $X_{0-}^{D^1, D^2}=x>0$; indeed, this can, for example, be seen by Corollary \ref{cor:value_goes_to_inf} (found below).
\end{remark}

%% file: g_function.tex
\section{Generalizing to a state-dependent profit rate}\label{sec:g-function}
\allowdisplaybreaks
In this section we consider a generalization of our game by broadening the definition of the reward functionals (cf. \eqref{eq_value-func} Definition \ref{def:value-funcs}) by allowing a state-dependent profit rate. In particular, we here re-define the expected reward to player $i$ as:
\begin{align}\label{g-function-payouts}
    J_{i}\left(x; D^1,D^2\right) :=
    \E_{x}\left[\int_{[0,\tau]} e^{-rt}g\left(X^{D^1,D^2}_{t-}\right)\circ d D_{t}^{i}\right],
\end{align}
where the function $g:[0,\infty) \rightarrow (0,\infty)$ is RCLL and we define:
\begin{align}\label{g-function-payouts2}
\begin{split}
     \int_{[0,\tau]} e^{-rt}g\left(X^{D^1,D^2}_{t-}\right)\circ d D_{t}^{i}:=
     &\int_{0}^{\tau}e^{-rt} \frac{g\left(X^{D^1,D^2}_{t-}\right) + g\left(X^{D^1,D^2}_{t-}-\right)}{2}d\left(D_{t}^{i}\right)^{c}\\
    +&\sum_{0\leq t\leq\tau}e^{-rt}\frac{\Delta D_{t}^{i}}{\sum_{j=1}^2\Delta D_t^{j}}
    \left(G\left(X^{D^1,D^2}_{t-}\right)-G\left(\left(X^{D^1,D^2}_{t-} - \sum_{j=1}^2\Delta D_t^{j}\right)\lor0\right)\right),\\
     G(x):= &\int_{0}^{x}g(u)du.
\end{split}
\end{align}
(Recall that Section \ref{sec:intro} discusses the relation of our definition \eqref{g-function-payouts2}  to previous literature). As in Definition \ref{def:value-funcs} we also define the reward functionals \eqref{g-function-payouts} to take the value $-\infty$ if the corresponding SDE \eqref{SDE} has no unique strong solution.

It it is easily seen that if $g\equiv1$, then  \eqref{g-function-payouts} collapses into our previous expected reward \eqref{eq_value-func}. 
\begin{remark}[Interpretation of the reward \eqref{g-function-payouts}]\label{reward-lit-rem-g} 
Let us give an interpretation the part:
    \begin{equation}\label{eq:state-dep-gpart-conta}
        \int_{0}^{\tau}e^{-rt}\frac{g\left(X^{D^1,D^2}_{t-}\right) + 
        g\left(X^{D^1,D^2}_{t-}-\right)}{2}d\left(D_{t}^{i}\right)^{c}
    \end{equation}
    of \eqref{g-function-payouts2}. In case $g$ is continuos,  \eqref{eq:state-dep-gpart-conta} collapses into $\displaystyle\int_{0}^{\tau}e^{-rt}g\left(X^{D^1,D^2}_{t-}\right)d\left(D_{t}^{i}\right)^{c}$
    which highlights the interpretation of \eqref{eq:state-dep-gpart-conta} as a state-dependent profit rate. Moreover, considering the special case that player $i$ acts alone so that the state process is reflected at a point $x^i$, in particular   $\left(D_{t}^{i}\right)=\left(D_{t}^{i}\right)^{c}=  (L_t^{x^i})$, reveals the meaning of the average in \eqref{eq:state-dep-gpart-conta}; in particular, the interpretation is in this case that the profit rate for reflection at $x^i$ is calculated using the average of $g\left(x^i\right)$ and $g\left(x^i-\right)$.
    
    Now suppose player $1$, say, controls with a jump $\Delta D_t^{1}>0$ at time $t-$ and that $\Delta D_t^{2}= 0$, then the value due to this jump for player $1$ is the discounted expected value of:
    \begin{equation*}
        G\left(X^{D^1,D^2}_{t-}\right)-
        G\left(X^{D^1,D^2}_{t-}-\Delta D_t^{i}\right)=
        \int_{X^{D^1,D^2}_{t-}-\Delta D_t^{1}}^{X^{D^1,D^2}_{t-}}g(u)du,
    \end{equation*}
    (recall that $\Delta D_t^{1}\leq X^{D^1,D^2}_{t-}$, by Definition \ref{def:Markovian-admissible-controls}). This highlights that \eqref{g-function-payouts2} corresponds to valuing lump sum extraction by integrating the profit rate over an interval corresponding to such an extraction. In general, we interpret the part: 
    \begin{equation*}
        G\left(X^{D^1,D^2}_{t-}\right)-
        G\left(\left(X^{D^1,D^2}_{t-} - \sum_{j=1}^2\Delta D_t^{j}\right)\lor0\right)
    \end{equation*}
    of \eqref{g-function-payouts2} as the total value of an aggregated extraction of the amount: 
    $$\displaystyle X^{D^1,D^2}_{t-}-\left(X^{D^1,D^2}_{t-}-\sum_{j=1}^2\Delta D_t^{j}\right)\lor0 = X^{D^1,D^2}_{t-} \wedge \sum_{j=1}^2\Delta D_t^{j}$$
    (i.e., the amount is capped in a way corresponding to it not being possible to extract more resources than are available), whereas the interpretation of the part $\displaystyle\frac{\Delta D_{t}^{i}}{\sum_{j=1}^2\Delta D_t^{j}}$ of \eqref{g-function-payouts2}, is that the total value extracted is divided among the players proportionally according to their attempted extraction amounts.
\end{remark}
 
\begin{remark}\label{remark:motivating-example}
An example for the profit rate $g$ in the expected reward \eqref{g-function-payouts}, and the associated integral $G$, is:
    \begin{equation}\label{eq:motivating-g}
        g(x) = \begin{cases}
            \frac{1}{2} & \text{ for: }  0 \leq x < \ell\\
            1 & \text{ for: }x \geq \ell,
        \end{cases}\quad 
        G(x) = \frac{x\land l}{2} + (x-l)\lor0, \enskip x \geq 0, \enskip \text{where $\ell\geq0$ is constant.}
    \end{equation}
    In particular, the case \eqref{eq:motivating-g} corresponds to there being a critical mass of resources $\ell$, dividing the state space into a low profitability region $(0,\ell)$ and a high profitability region $[\ell,\infty)$; see Example \ref{Example:g_jump} for an equilibrium analysis for this particular, and related, $g$ function(s).
\end{remark}
We are now ready to present the equilibrium definition, which is completely analogous to Definition \ref{def:Markov-NE}.
\begin{definition}[Global Markovian Nash equilibrium]\label{def:Markov-NE-for-g}
    A pair of admissible control strategies $(D^1,D^2)\in\mathbb{L}\times\mathbb{L}$ with corresponding controlled SDE \eqref{SDE} is said to be a global Markovian Nash equilibrium for the expected rewards given by \eqref{g-function-payouts} if: 
    \begin{align*}
        J_1\left(x,D^{1},D^{2}\right) & \geq J_1\left(x,D,D^{2}\right)\\
        J_2\left(x,D^{1},D^{2}\right) & \geq J_2\left(x,D^{1},D\right),
    \end{align*}
    for all admissible control strategies $D\in\mathbb{L}$, and all $x \geq 0$.
\end{definition}

We find the following trivial equilibria (analogous to those in Theorem \ref{thm:trivialNE}). 
\begin{theorem}\label{thm:trivialNE-g} 
    A pair of admissible control strategies $(D^1, D^2)\in \mathbb{L}\times\mathbb{L}$ satisfying $B^{i} = [0, \infty)$ and ${\cal J}^{i}(x)=x$ (which implies that $D^i_0= x$ and $\tau=0$ a.s., for any $X_{0-}=x$), is a global Markovian Nash equilibrium. The corresponding equilibrium values are given by:
    \begin{equation*}
        V(x)= \frac{G(x)}{2}, \enskip x \geq 0.
    \end{equation*}
\end{theorem}
\begin{proof}
This follows from  Definition \ref{def:Markov-NE-for-g} and the definition in \eqref{g-function-payouts}--\eqref{g-function-payouts2} (recalling also that we have $g\geq0$).
\end{proof}

\subsection{Equilibria with symmetric extraction rates and skew points above a threshold}\label{subsec:g-function-better-ne}
We will now show that the game of the present section has global Markovian Nash equilibria of symmetric threshold type in the sense that the players extract with the same rates and skew points and only above a boundary $b$, under suitable conditions. The main results are Theorem \ref{Theorem_verification_g_NE_2} (equilibrium verification) 
and Theorem \ref{Theorem_existence_g_NE_1} (equilibrium existence).

Let us start with some motivating observations. In analogy with Section \ref{sec:Naive_Lebesgue_absolutely_continuous_Nash_equilibria} we make the ansatz that the value function is on the form: 
\begin{equation}\label{eq:newVb}
    V_{b}(x) = \begin{cases}
        \frac{g(b)}{\psi'(b)}\psi(x) & \text{ for: } 0 \leq x<b\\
        \frac{g(b)}{\psi'(b)}\psi(b) + G(x) - G(b) & \text{ for: }x\geq b,
    \end{cases} 
\end{equation}
where $b\geq0$  is some threshold value to be determined, and $\psi$ is the fundamental increasing solution of the ODE corresponding to the uncontrolled SDE (defined in the beginning of Section \ref{sec:Naive_Lebesgue_absolutely_continuous_Nash_equilibria}).

Supposing for the moment that $g$ is differentiable (this is not assumed for any of the results in this section) we have that arguments similar to those in Section \ref{sec:Naive_Lebesgue_absolutely_continuous_Nash_equilibria} yield that we should try to find an equilibrium control rate in line with \eqref{HJB-first-thing} when also adjusting for the presence of the profit rate $g$. In particular, this yields that the equilibrium candidate value function $V_b$ in \eqref{eq:newVb} and a corresponding threshold rate $\lambda^*_b$
should satisfy:
\begin{align}\label{HJB:secind-thing}
\begin{split}
    \frac{\sigma^2(x)}{2}V_b''(x) + (\mu(x)-2\lambda_b^\ast(x))V_b'(x)-
    rV_{b}(x)+ \lambda_b^\ast(x)g(x)&=0,\enskip\text{for: }x\neq b\\
    \frac{\sigma^2(x)}{2}V_b''(x) + (\mu(x)-\lambda_b^\ast(x)-\lambda)V_b'(x)-
    rV_{b}(x) + \lambda g(x)&\leq0,\enskip\text{for: }x\neq b\text{ and all: }\lambda\geq0.
\end{split}
\end{align}
Given the ansatz $V_b'\geq g$ and the observation that $V_b'(x)=g(x),x\geq b$, this immediately yields, as in Section \ref{sec:Naive_Lebesgue_absolutely_continuous_Nash_equilibria}, a candidate equilibrium rate:
\begin{equation}\label{eq:lambdastar-V_b}
    \lambda_b^\ast(x) 
    = \1_{\{x>b\}}\frac{A_X V_{b}(x) - rV_{b}(x)}{g(x)}
    = \1_{\{x>b\}}\frac{A_X G(x) - rG(x)- 
    r\left(\frac{g(b)}{\psi'(b)}\psi(b)- G(b)\right)}{g(x)}\geq0.
\end{equation}
(The above expression will be non-negative by assumption, cf. Definition \ref{def:Admissible-g} below).

Now suppose $(X_t)$ is the state process controlled by one player using the control rate \eqref{eq:lambdastar-V_b}. Using It\^{o}'s formula it can be seen that the process $e^{-rt}V_b(X_t)$ is a (local) martingale on $0\leq t\leq\tau$.

Let us now consider the case that $G$ is not smooth. Our ansatz for the value function $V_b$ in this case is that the (local) martingale property of $e^{-rt}V_b\left(X_{t}\right)$ on $0\leq t\leq\tau$ is preserved, where $(X_t)$ is again the process resulting from one player using a candidate equilibrium strategy. Furthermore, our ansatz is that discontinuities in $g$ (on the interval $(b,\infty)$, where the players act) will correspond to skew points in the equilibrium strategies, so that $(X_t)_{t\geq0}$ is on the form:
\begin{equation*}
    dX_t = \left(\mu(X_t)-\lambda_b^\ast(X_t)\right)dt + 
    \sigma(X_t)dW_t - \sum_j c_j dL_t^{x_j}(X),
\end{equation*}
which we note is in line with our definition of admissible control strategies (Definition \ref{def:Markovian-admissible-controls}). (See the proof of Theorem \ref{Theorem_verification_g_NE_2} for a discussion of existence of solutions to SDEs of this kind). For a candidate threshold $b$, we will assume (Definition \ref{def:Admissible-g}) that the set of discontinuity points of $g$ which exceed $b$ (strictly), denoted by $\{l_1,\ldots,l_q\}$, is finite (for ease of notation we do not index these points by $b$). By It\^{o}'s formula (see e.g., \cite[Sec 3.5]{Peskir}) we have:
\begin{align}\label{eq:deriv_martingale}
\begin{split}
    e^{-rt}V_b(X_{t}) &= 
    V_b(x) + \int_{0}^{t}e^{-rs}
    \left(A_{X}V_b(X_{s})-\lambda_b^\ast(X_{s})V_b'(X_{s})-rV_b(X_{s})\right)
    \1_{\left\{X_{s}\notin\{l_{1},\ldots, l_{q}\}\right\}}ds\\
    &+\int_{0}^{t}e^{-rs}V_{b}'(X_{s})\sigma(X_{s})
    \1_{\left\{X_{s}\notin\{l_{1},\ldots, l_{q}\}\right\}}dW_{s}-
    \int_{0}^{t}e^{-rs}\sum_{j}\frac{g(x_j+)+g(x_j-)}{2}c_j dL^{x_{j}}_{s}(X)\\
    &+\int_{0}^{t}e^{-rs}
    \sum_{j}\frac{g(l_{j}+)-g(l_{j}-)}{2}dL^{l_{j}}_{s}(X),
    \enskip 0 \leq t\leq\tau.
\end{split}
\end{align}
Recall that we want the left hand side in \eqref{eq:deriv_martingale} to be a (local) martingale. With $\lambda_b^\ast$ defined in line with \eqref{eq:lambdastar-V_b} (on discontinuity points of $g$ we set $\lambda_b^\ast=0$, cf. \eqref{eq:lambda-inthmwithg} below) we obtain that the first integral in the right hand side of \eqref{eq:deriv_martingale} is zero (use $V_b'(x)=g(x),x\geq b$, $\lambda_b^\ast(x)=0,x\leq b$ and \eqref{eq:newVb} to see this). Moreover, in order for the local time integrals to vanish (to the end of obtaining a martingale), it is clear that we also need $x_j := l_j$ and for the skew point intensities to satisfy:
\begin{equation}\label{eq:def-c_j}
    c_j := \frac{g(l_{j}+)-g(l_{j}-)}{g(l_{j}+)+g(l_{j}-)},  \enskip j=1,\ldots,q.
\end{equation}
In order to formalize the motivational observations above properly we need to impose the following conditions on the profit rate function $g$ for any given $b$.
\begin{definition}[$b$-admissibility for profit rates $g$]\label{def:Admissible-g}
    A function $g:[0, \infty)\rightarrow\left(0,1\right]$ is said to be $b$-admissible for a fixed threshold $b\geq0$, if it is a RCLL function satisfying the following conditions:
    \begin{itemize}
        \item There exists a point $y_1\geq0$ with $g(y_1)=g(x)$ for $x>y_1$. 
        \item The set of discontinuity points of $g$ which are strictly greater than $b$, denoted by $\{l_1,\ldots,l_q\}$, is finite.
        \item $g$ is continuously differentiable on $\mathbb{R}\setminus\Theta$ where $\Theta$ is a finite set satisfying $\Theta\supseteq\{l_{1},\ldots, l_{q}, b\}$.
        \item Each $c_j$, see \eqref{eq:def-c_j}, satisfies $c_j \in (0, 1/2]$, i.e., $g(l_j +)\leq 3g(l_j -)$ and  $\Delta g(l_j):=g(l_j +)-g(l_j -)>0$.
        \item The function $(b,\infty)\setminus\Theta\ni x\mapsto\frac{g'(x)}{g(x)}$ is bounded.
    \end{itemize}
\end{definition}
Given that Definition \ref{def:Admissible-g} allows for non-smoothness we  modify the proposed extraction rate in \eqref{eq:lambdastar-V_b} to:
\begin{equation}\label{eq:lambda-inthmwithg}
    \lambda_b^\ast(x): = \frac{A_{X}V_b(x)-rV_b(x)}{g(x)}
    \1_{\{x \in (b, \infty)\setminus\Theta\}} = 
    \frac{A_X G(x) - rG(x)- r\left(\frac{g(b)}{\psi'(b)}\psi(b)- G(b)\right)}{g(x)}
    \1_{\{x \in (b, \infty)\setminus\Theta\}}.
\end{equation}

\begin{remark}\label{rem:generalize-to-big-g} 
    The assumption that $g$ takes values $(0,1]$ is made for ease of exposition. Indeed, this assumption can be generalized so that $g$ is allowed to take values in $(0,M]$ for an arbitrary constant $M \in (0,\infty)$. This is clear in view of the equilibrium definition and the expected reward functionals; to see this consider the standardization $g/M$.
\end{remark}

We are now ready to state equilibrium verification and existence results for the game in the present section. Their proofs are found in Sections \ref{subsec:proof_g_verification} and \ref{subsec:proof_g_existence}, respectively.  

\begin{theorem}[Verification]\label{Theorem_verification_g_NE_2}
    Suppose Assumptions \ref{Assumption_1} and \ref{Assumption_2} hold. Consider a $b$-admissible profit rate $g$, for a fixed $b\geq0$. Suppose the following conditions hold:
    \begin{align}
        \label{NE:cond2-g}
        V_b\text{ defined in \eqref{eq:newVb} satisfies } V_b'(x)\geq g(x),x \in [0,\infty)\setminus\Theta,\\
        \lambda_b^\ast\text{ defined in \eqref{eq:lambda-inthmwithg} satisfies }
        \lambda_{b}^{\ast}(x)\geq0,x\geq0.\label{NE:cond1-g}
    \end{align}
    Then the corresponding game has a (symmetric) global Markovian Nash equilibrium with value function $V_{b}$ and equilibrium strategies (cf. Definition \ref{def:Markovian-admissible-controls}):
    \begin{equation*}
        D^{1} = D^{2} = 
        \left\{\lambda_b^\ast, \{l_{1}, \ldots, l_q\}, \{c_{1}, \ldots, c_q\}, 
        \emptyset, \emptyset\right\}.
    \end{equation*}
    (Recall that $\{l_1,\ldots,l_q\}$ are the  discontinuity points of $g$ which exceed $b$ (strictly), and that $c_i,i=1,...,q$ are defined in \eqref{eq:def-c_j}.)
\end{theorem}
Note that the equilibrium strategies of Theorem \ref{Theorem_verification_g_NE_2} correspond to each player controlling with a combination of a classical state-dependent rate and skew points (the associated SDE is given in in \eqref{SDE-g-function-proof}, in the proof of Theorem \ref{Theorem_verification_g_NE_2}, in Section \ref{subsec:proof_g_verification}, below). 

\begin{theorem}[Equilibrium existence]\label{Theorem_existence_g_NE_1}
    Suppose Assumptions \ref{Assumption_1} and \ref{Assumption_2} hold. Let $g$  be a $b'$-admissible function, for some $b'\geq0$ (which means that $g$ is $b$-admissible for any $b\geq b'$).
    \begin{enumerate}[(I)]
        \item Suppose additionally that $g(y_1)\geq g(x), x \geq 0$ (for $y_1$ given in Definition \ref{def:Admissible-g}). Then, there exists a constant  $\underline{b}\geq b'$ such that conditions \eqref{NE:cond2-g}--\eqref{NE:cond1-g} hold (and we have a global Markovian Nash equilibrium as in Theorem \ref{Theorem_verification_g_NE_2}), for any $b\geq \underline{b}$. 
        \item Consider a fixed $b \geq b'$. If $x\mapsto\mu(x)-rx$ is non-decreasing on $[0, b)$ with $\mu(0)\geq0$ and either $g(b)=1$ or $g$ is non-decreasing on $[0, b)$, then \eqref{NE:cond2-g} holds.
        Hence, if also \eqref{NE:cond1-g} holds, then we have a global Markovian Nash equilibrium as in Theorem \ref{Theorem_verification_g_NE_2} for such $b$. 
        \item Suppose additionally that $g$ is non-decreasing. If $x\mapsto\mu(x)-rx$ is non-decreasing on $[0, \infty)$ with $\mu(0)\geq0$
        then \eqref{NE:cond2-g} and \eqref{NE:cond1-g} hold (and we thus have a global Markovian Nash equilibrium as in Theorem \ref{Theorem_verification_g_NE_2}), for any $b \geq b'$. 
    \end{enumerate}
\end{theorem}
The following result is analogous to Corollary \ref{cor:value_goes_to_inf}. 

\begin{corollary}\label{cor:value_goes_to_inf_g}
    Suppose all assumption of Theorem \ref{Theorem_existence_g_NE_1}(I) hold. Then, there exists a constant $\underline{b}$ such that we have a global Markovian Nash equilibrium for each $b \geq \underline{b}$ with an equilibrium value $V_b(x)$ that is strictly increasing in $b$ on $(\underline{b},\infty)$ for fixed $x>0$, with:
        \begin{equation*}
            \lim_{b\rightarrow\infty}V_b(x)=\infty.
        \end{equation*}
\end{corollary}
\begin{proof}
    Use Theorem \ref{Theorem_existence_g_NE_1}(I) to see that we can select $\underline{b}$ with $\underline{b}>b_{2}, b', y_{1}$ (see Theorem \ref{Theorem_existence_g_NE_1} and Lemma \ref{Lemma_that_I_don't_know_what_to_call} for definitions of these constants) so that we have a global Markovian Nash equilibrium for each  $b \geq \underline{b}$.  Now use $g'(b)=0, b >y_1$ and Lemma \ref{Lemma_that_I_don't_know_what_to_call}(II) to see, using some calculations, for any fixed $x>0$, and for $ b\geq \underline{b}$ that the equilibrium value $V_b$ (given by \eqref{eq:newVb}) satisfies:
    \begin{equation*}
            \frac{\partial V_{b}(x)}{\partial b} = -\frac{g(b)\psi(x\land b)\psi''(b)}{\psi'(b)^{2}}>0,
    \end{equation*}
    which implies that $b\mapsto V_b(x)$ is strictly increasing on the relevant interval. The last result follows from the form of $V_b$ in \eqref{eq:newVb} and Lemma \ref{Lemma_that_I_don't_know_what_to_call}(III). 
\end{proof}
\begin{remark}\label{Remark:Two-equilibria}
    It is worth noting that in the equilibria described in Theorem \ref{Theorem_verification_g_NE_2} no player chooses a strategy with a skew point at the point $b$, even if the corresponding profit rate function $g$ has a discontinuity at this point. However, there is another Nash equilibrium where the players use skew points at $b$ in the case that $b$ is a discontinuity point of $g$. In particular, consider the value function (cf. \eqref{eq:newVb}): 
        \begin{equation*}
            V_{b-}(x) := \begin{cases}
                \frac{g(b-)}{\psi'(b)}\psi(x) & \text{ for: } 0 \leq x<b\\
                \frac{g(b-)}{\psi'(b)}\psi(b) + G(x) - G(b) & \text{ for: }x\geq b,
            \end{cases}
        \end{equation*}
    and the symmetric pair of strategies given by the extraction rate  $\displaystyle x\mapsto \frac{A_{X}V_{b-}(x)-rV_{b-}(x)}{g(x)}\1_{\{x \in (b, \infty)\setminus\Theta\}}$ and skew points at discontinuity points of $g$ which are greater than or equal to $b$ with intensities determined by equation \eqref{eq:def-c_j}. Then, we have corresponding verification and existence results analogous to those in the present section with proofs analogous to those in Sections \ref{subsec:proof_g_verification}--\ref{subsec:proof_g_existence} with the equilibrium value function $V_{b-}$.
\end{remark}

\subsection{Proof of Theorem \ref{Theorem_verification_g_NE_2}}\label{subsec:proof_g_verification}
It is immediately verified that the strategies $D^1$ and $D^2$ are admissible (i.e., belonging to $\mathbb{L}$). It suffices to show that one of the players, say player $1$, does not gain anything by deviating from $(D^1,D^2)$, for any given $X_{0-}=x>0$ i.e.:
\begin{equation}\label{claim1}
    J_1\left(x;D, D^2\right)\leq V_b(x),
\end{equation}
for any $D\in \mathbb{L}$, and that the expected reward corresponding to the strategy tuple $(D^1, D^2)$ satisfies: 
\begin{equation}\label{claim2}
    J_1\left(x;D^1, D^2\right)=V_b(x),
\end{equation}
where $V_b$ is given by \eqref{eq:newVb}, cf. Definition \ref{def:Markov-NE-for-g} (the case $X_{0-}=0$ is trivial); since the same results for the other player can be proved in the same way. 
    
Let us prove the first claim (cf. \eqref{claim1}). Consider an arbitrary $X_{0-}=x>0$ and suppose that (only) player $1$ deviates from $(D^1, D^2)$ by using an arbitrary admissible control strategy $D\in \mathbb{L}$. The associated SDE is:
\begin{equation}\label{eq:g-proof-controlled-sde-deviation}
    dX_{t} = \left(\mu(X_{t})-\lambda_b^\ast(X_{t})\right)dt + \sigma(X_{t})dW_{t}-
    \sum_{j}c_{j}dL_{t}^{l_{j}}(X) - dD_t,\enskip X_{0-}=x , \enskip 0 \leq t \leq \tau.
\end{equation}
Without loss of generality assume that the controlled SDE \eqref{eq:g-proof-controlled-sde-deviation} has a unique strong solution (since if the control strategy $D$ is such that there is no unique strong solution then the corresponding expected payoff for player $1$ is $J_{1}(x,D, D^2)=-\infty$, and it is hence clear that the desired conclusion holds).
    
Using $\tau_{n}:=\inf\{t\geq0:X_{t}\geq n\}$ and a generalized It\^{o} formula (cf. e.g., \cite[Sec. 3.5]{Peskir} and \cite[Theorem 3.1]{peskir2007change}) together with the smoothness properties of $V_b$ (see \eqref{eq:newVb} and Definition \ref{def:Admissible-g}), and \eqref{eq:g-proof-controlled-sde-deviation} we have:
\begin{align*}
    e^{-r(\tau_{n}\land\tau\land n)}V_b(X_{(\tau_{n}\land\tau\land n)-})
    &=V_b(x) +\int_{0}^{\tau\land\tau_{n}\land n}e^{-rs}
    V'_{b}(X_{s})\sigma(X_{s})\1_{\{X_s \notin  \Theta\}}dW_{s}\\
    &+\int_{0}^{\tau\land\tau_{n}\land n}e^{-rs}
    \left(\frac{1}{2}\sigma^2(X_s)V''_{b}(X_{s})+
    (\mu(X_s)-\lambda_b^\ast(X_s))V'_{b}(X_{s}) -rV_b(X_s)\right)
    \1_{\{X_s \notin\Theta\}}ds\\
    &+ \int_{0}^{\tau\land\tau_{n}\land n}e^{-rs}
    \sum_{j}\frac{V'_{b}(l_j+)-V'_{b}(l_j-)}{2}dL^{l_{j}}_{s}(X)\\
    & -\int_{0}^{\tau\land\tau_{n}\land n}e^{-rs}\sum_{j}
    \frac{V'_{b}(l_j+)+V'_{b}(l_j-)}{2}c_{j}dL_{s}^{l_{j}}(X)\\
    &-\int_{0}^{\tau\land\tau_{n}\land n}e^{-rs}
    \frac{V'_{b}(X_{s-}+)+V'_{b}(X_{s-}-)}{2}dD_s^c\\
    & - \sum_{0\leq s<\tau_n\land\tau\land n}e^{-rs}
    \left(V_{b}(X_{s-})-V_{b}(X_{s-} - \Delta D_s)\right).
\end{align*}
Now use \eqref{eq:newVb}, \eqref{eq:def-c_j}, and that $b<l_j$ to see that:  
\begin{equation}\label{eq:proof-skew-point-calculation1}
    \frac{V_b'(l_{j}+)-V_b'(l_{j}-)}{2}-
    \frac{V_b'(l_{j}+)+V_b'(l_{j}-)}{2}c_{j}= 0,
\end{equation}
which implies the integrals with respect to local time above sum up to zero and vanish. Now use \eqref{eq:newVb} and \eqref{eq:lambda-inthmwithg} to find that:
\begin{align}\label{eq:proof-hjb-help}
\begin{split}
    \left(\frac{1}{2}\sigma^2(y)V''_{b}(y)+
    (\mu(y)-\lambda_b^\ast(y))V'_{b}(y) -rV_b(y)
    \right)\1_{\{y \notin  \Theta\}}
    &= \left(A_XV_b(y) -rV_b(y)- \lambda_b^\ast(y)V'_{b}(y)\right)\1_{\{y \notin  \Theta\}}\\
    = \left(A_XV_b(y) -rV_b(y)- \lambda_b^\ast(y)g(y)
    \right)\1_{\{y \notin  \Theta\}}&=0. 
\end{split}
\end{align}
Using these observations, taking expectation in the It\^{o} formula expression above and using the condition \eqref{NE:cond2-g} (as well as \eqref{eq:newVb}) we obtain:
\begin{align}\label{eq:ito-in-the-proof-g}
    V_b(x) &=
    \mathbb{E}_x\left[e^{-r(\tau_{n}\land\tau\land n)}
    V_b(X_{(\tau_{n}\land\tau\land n)-})\right]+
    \E_x\left[\int_{0}^{\tau\land\tau_{n}\land n}e^{-rs}
    \frac{V'_{b}(X_{s-}+)+V'_{b}(X_{s-}-)}{2}dD_s^{c}\right]\nonumber\\
    &+\E_x\left[\sum_{0\leq s<\tau_n\land\tau\land n}e^{-rs}
    \left(V_{b}(X_{s-})-V_{b}(X_{s-} - \Delta D_s)\right)\right]\nonumber\\
    &\geq\mathbb{E}_x\left[e^{-r(\tau_{n}\land\tau\land n)}
    V_b(X_{(\tau_{n}\land\tau\land n)-})\right]+
    \E_x\left[\int_{0}^{\tau\land\tau_{n}\land n}e^{-rs}
    \frac{g(X_{s-}+)+g(X_{s-}-)}{2}dD_s^c\right]\nonumber\\
    &+\E_x\left[\sum_{0\leq s<\tau_n\land\tau\land n}e^{-rs}
    \int_{X_{s-}-\Delta D_s}^{X_{s-}}g(u)du\right]\nonumber\\
    &\geq\mathbb{E}_x\left[e^{-r(\tau_{n}\land\tau\land n)}
    \left(G(X_{(\tau_{n}\land\tau\land n)-}) - G\left(\left(X_{(\tau_{n}\land\tau\land n)-} - \Delta D_{\tau_{n}\land\tau\land n}\right)\lor0\right)\right)\right]\nonumber\\
    &+\E_x\left[\int_{0}^{\tau\land\tau_{n}\land n}
    e^{-rs}\frac{g(X_{s-})+g(X_{s-}-)}{2}dD_s^c +
    \sum_{0\leq s<\tau_n\land\tau\land n}e^{-rs}
    \left(G(X_{s-})-G((X_{s-} - \Delta D_s)\lor0)\right)\right]\\
    &=\E_x\left[\int_{0}^{\tau\land\tau_{n}\land n}
    e^{-rs}\frac{g(X_{s-})+g(X_{s-}-)}{2}dD_s^c +
    \sum_{0\leq s\leq\tau_n\land\tau\land n}e^{-rs}
    \left(G(X_{s-})-G((X_{s-} - \Delta D_s)\lor0)\right)\right]\nonumber\\
    &\overset{n\rightarrow\infty}{\longrightarrow}
    \E_x\left[\int_{0}^{\tau}e^{-rs}\frac{g(X_{s-})+g(X_{s-}-)}{2}dD_s^c +
    \sum_{0\leq s\leq\tau}e^{-rs}
    \left(G(X_{s-})-G((X_{s-} - \Delta D_s)\vee 0)\right)\right]\nonumber\\
    & = J_1\left(x; D, D^2\right),
    \nonumber
\end{align}
where the second to last step holds due to monotone convergence theorem and $\tau_{n}\land\tau\land n\overset{n\rightarrow\infty}{\longrightarrow}\tau$ a.s. and the property $\Delta D_t\leq X_{t-}$ and the last step holds by definition (use \eqref{g-function-payouts}--\eqref{g-function-payouts2} and the observation that $\Delta D^2_t=0$). We have thus proved the first claim.
    
Let us prove the second claim (cf. \eqref{claim2}). Consider the case that no player deviates from $(D^1,D^2)$. The associated SDE:
\begin{equation}\label{SDE-g-function-proof}
    dX_{t} = \left(\mu(X_{t})-2\lambda_b^\ast(X_{t})\right)dt + \sigma(X_{t})dW_{t}-
    \sum_{j}2c_{j}dL_{t}^{l_{j}}(X),\enskip X_{0-}=x, \enskip 
    0\leq t \leq \tau, 
\end{equation}
has a unique strong solution, which can be seen as follows: First note that strong unique existence until $\tau_n=\{t\geq0:X_{t}\geq n\}$ for any $n$ can be verified (see e.g., \cite[Theorem 4.3, Corollary 2.5 and Remark 2.6]{bass2005one}  and the occupation time formula; note that $\lambda^\ast_b$ is locally bounded by Definition \ref{def:Admissible-g}).  Now note that the dynamics of the SDE on $(n,\infty)$, for $n>y_1,b,l_1,\ldots,l_q$, have no skew points and also Lipschitz continuos coefficients (cf. e.g., Definition \ref{def:Admissible-g}, Assumption \ref{Assumption_1} and \eqref{eq:lambda-inthmwithg}). Now standard arguments of piecing together solutions on suitable subintervals yield the desired conclusion.

From now on in the proof $(X_t)$ denotes the solution to the SDE \eqref{SDE-g-function-proof}. Using the generalized It\^{o} formula as above we obtain (using \eqref{SDE-g-function-proof}):
\begin{align}\label{ito-for-other-proof}
    \nonumber
    e^{-r(\tau_{n}\land\tau\land n)}V_b(X_{\tau_{n}\land\tau\land n}) &= 
    V_b(x) +\int_{0}^{\tau_{n}\land\tau\land n}
    e^{-rs}V'_{b}(X_{s})\sigma(X_{s})\1_{\{X_{s}\notin\Theta\}}dW_{s}\\
    \nonumber
    &+\int_{0}^{\tau\land\tau_{n}\land n}e^{-rs}
    \left(\frac{1}{2}\sigma^2(X_s)V''_{b}(X_{s})+
    (\mu(X_s)-2\lambda_b^\ast(X_s))V'_{b}(X_{s}) -rV_b(X_s)\right)
    \1_{\{X_s \notin  \Theta\}}ds\\
    &+\int_{0}^{\tau_{n}\land\tau\land n}e^{-rs}
    \sum_{j}\frac{V'_{b}(l_j+)-V'_{b}(l_j-)}{2}dL^{l_{j}}_{s}(X)\\
    \nonumber
    &-\int_{0}^{\tau_{n}\land\tau\land n} e^{-rs}
    \sum_{j}\frac{V'_{b}(l_j+)+V'_{b}(l_j-)}{2}2c_j dL^{l_{j}}_{s}(X).
\end{align}
Now note (as in \eqref{eq:proof-skew-point-calculation1}) that \eqref{eq:def-c_j} is equivalent to:
\begin{equation}\label{eq:proof-skew-point-calculation2}
    \frac{V'_b(l_j+)-V'_b(l_j-)}{2} - \left(V'_{b}(l_j+)+V'_{b}(l_j-)\right)c_j=
    -c_{j}\frac{g(l_{j}+)+g(l_{j}-)}{2},
\end{equation}
and that (cf. \eqref{eq:newVb}, \eqref{eq:lambdastar-V_b} and \eqref{eq:proof-hjb-help}):
\begin{align*}
    \left(\frac{1}{2}\sigma^2(y)V''_{b}(y)+
    (\mu(y)-2\lambda_b^\ast(y))V'_{b}(y) -rV_b(y)
    \right)\1_{\{y \notin  \Theta\}}&= 
    -\lambda_b^\ast(y)V'_{b}(y) \1_{\{y \notin  \Theta\}}\\
    &= -\lambda_b^\ast(y))g(y) \1_{\{y \notin  \Theta\}}. 
\end{align*}
Using these observations and taking expectation on both sides in \eqref{ito-for-other-proof} we obtain:
\begin{align*}
    V_b(x) 
    &= \E_x\left[e^{-r(\tau_{n}\land\tau\land n)}
    V_b(X_{\tau_{n}\land\tau\land n})\right]+
    \E_x\left[\int_{0}^{\tau_{n}\land\tau\land n}
    e^{-rs}\lambda_b^\ast(X_{s})g(X_{s})ds\right]\\
    &+\E_x\left[\int_{0}^{\tau_{n}\land\tau\land n}e^{-rs}
    \sum_{j}c_j\frac{g(l_{j}+)+g(l_{j}-)}{2}dL^{l_{j}}_{s}(X)\right].
\end{align*}
Hence, proving that:
\begin{align}\label{proofhelp}
    \lim_{n\rightarrow\infty}\E_x\left[e^{-r(\tau_{n}\land\tau\land n)}
    V_b(X_{\tau_{n}\land\tau\land n})\right]=0,
\end{align}
proves also that $V_b(x)= J_{1}\left(x; D^{1}, D^{2}\right)$ (i.e., the second claim); to see this use monotone convergence and that the controlled process $(X_t)$, given by \eqref{SDE-g-function-proof}, has continuos paths and the control process of player $1$ (as well as player $2$) is
\begin{align*}
    D^1_t = \int_{0}^{t} \lambda_b^\ast(X_{s}) ds+ \sum_{j}c_{j} L^{l_{j}}_{s}(X), 
\end{align*}
so that:
\begin{equation*}
    \int_{0}^{\tau}e^{-rs}\lambda_b^\ast(X_{s})g(X_{s})ds+
    \int_{0}^{\tau}e^{-rs}\sum_{j}c_{j}\frac{g(l_{j})+g(l_{j}-)}{2}dL^{l_{j}}_{s}(X) =
    \int_0^\tau e^{-rt}g\left(X^{D^1,D^2}_{t-}\right)\circ d D_{t}^1.
\end{equation*}
It remains to prove that \eqref{proofhelp} holds and for this we note that:
\begin{equation*}
    0 \leq  \limsup_{n\rightarrow\infty}\E_x\left[e^{-r(\tau_{n}\land\tau\land n)}
    V_{b}(X_{\tau_{n}\land\tau\land n})\right]\leq
    \limsup_{n\rightarrow\infty}
    \E_x\left[e^{-r(\tau_{n}\land\tau\land n)}AX_{\tau_{n}\land\tau\land n} \right],
\end{equation*}
for some constant $A>0$ (cf. \eqref{eq:newVb}, Definition \ref{def:Admissible-g}, and that $V'_b$ is bounded and $V_{b}(0)=0$). Hence, to prove \eqref{proofhelp} it suffices to prove:
\begin{align}\label{g-func-the-thing-that-needs-to-die}
    \limsup_{n\rightarrow\infty}
    \E_x\left[e^{-r(\tau_{n}\land\tau\land n)}X_{\tau_{n}\land\tau\land n}\right]
    \leq \limsup_{n\rightarrow\infty}
    \E_x\left[e^{-r(\tau_{n}\land n)}n\right]=0,
\end{align}
where inequality above follows from the observation that $X_\tau=0$ on $\{\tau<\infty\}$. Hence, all all that remains to do is to show that the equality in \eqref{g-func-the-thing-that-needs-to-die} holds.

Pick a constant $\kappa'$ with $\kappa'>\kappa$ (cf. Assumption \ref{Assumption_2}) and $\kappa'>\theta$ for all $\theta\in\Theta$ as well as $\kappa'>y_1, b$ (cf. Definition \ref{def:Admissible-g}) so that $g(y) = g(\kappa')$ (and $g'(y)=0$, used below) for $y>\kappa'$. For $y\geq\kappa'$, we then have that the drift function in \eqref{SDE-g-function-proof} satisfies:
\begin{align}\label{g-proof-set-up-for-comparison-principle}
\begin{split}
    \mu(y) - 2\lambda_b^\ast(y)
    &= \mu(y) -2 \frac{\frac{\sigma^2(y)}{2}g'(y)+\mu(y)g(y)-rV_b(y)}{g(y)}\\
    &= 2r\frac{V_{b}(y)}{g(y)}-\mu(y)\\
    &= 2r\left(\frac{\frac{g(b)}{\psi'(b)}\psi(b) + G(\kappa') - G(b)}{g(\kappa')} + 
    \int_{\kappa'}^{y}\frac{g(u)}{g(\kappa')}du\right) - \mu(y)\\
    &= 2ry + 2r\left(\frac{\frac{g(b)}{\psi'(b)}\psi(b) + G(\kappa') - G(b)}{g(\kappa')}-\kappa'\right)- \mu(y).
\end{split}
\end{align}
Now consider the SDE:
\begin{align}\label{eq:sde-g=1-no-dev}
\begin{split}
    d\xi_t &= \mu_{\kappa'}(\xi_t)dt+
    \sigma(\xi_t)dW_t, \enskip \xi_0 = x,\enskip 0 \leq t \leq 
    \tau^\xi:=\inf\{t\geq0: \xi_t\leq 0\}\\
    \mu_{\kappa'}(y) &:=
    2ry+ \mathcal{C} - \mu(y),
\end{split}
\end{align}
where we select the constant  $\mathcal{C}$ so that:
\begin{equation}\label{help-asdas}
    \mu_{\kappa'}(y)\geq \mu(y)-2\lambda_{b}^{\ast}(y),\enskip  y\geq0.
\end{equation}
(It is easily seen, using e.g., continuity of $\mu$, that such a constant $\mathcal{C} $ can be found, but its exact value is not useful here). This SDE has a unique strong solution $(\xi_t)_{t\geq0}$ (by Lipschitz continuity), and in the below we will prove that: 
\begin{equation}\label{xi_needs_to_die}
    \limsup_{n\rightarrow\infty}
    \E_x\left[e^{-r(\tau^\xi_{n} \land n)}n\right]=0,
\end{equation}
where $\tau_n^\xi=\{t\geq0:\xi_{t}\geq n\}$. Using the comparison principle we immediately find that $\tau_{n}\geq\tau_{n}^{\xi}$ for $n\geq x$ a.s. so that \eqref{xi_needs_to_die} implies the inequality \eqref{g-func-the-thing-that-needs-to-die} (note that we for the comparison principle use \eqref{help-asdas} and that the skew point intensities in \eqref{SDE-g-function-proof} are negative, i.e. $-c_j<0$, cf. e.g., \cite[Theorems 3.1 and 4.6]{bass2005one}, and \cite[Ch. IX]{revuz2013continuous}).

It remains to show that \eqref{xi_needs_to_die} holds. Let $\psi_{\kappa'}$ be the increasing solution to the differential equation associated with \eqref{eq:sde-g=1-no-dev}:
\begin{equation}\label{ODE-psib}
    \frac{\sigma^{2}(y)}{2}\psi_{\kappa'}''(y)+ \mu_{\kappa'}(y)\psi_{\kappa'}'(y)=r\psi_{\kappa'}(y), \enskip \psi_{\kappa'}'(0)=1 \enskip, \psi_{\kappa'}(0)=0,
\end{equation}
and note (see e.g., \cite[II.10]{borodin2012handbook}) that, for $x\leq n$, we have: 
\begin{equation*}
    \frac{\psi_{\kappa'}(x)}{\psi_{\kappa'}(n)} = \E_x\left[e^{-r\tau_{n}^{\xi}}\right].
\end{equation*}
Using this and basic observations we find, for $x\leq n$, that:
\begin{align}\label{eq:generalized-psi=expec-bound}
\begin{split}
    0 & \leq \E_x\left[e^{-r(\tau^\xi_{n} \land n)}n\right]\\
    &\leq n\E_x\left[e^{-r\tau_{n}^{\xi}}\right] + ne^{-rn}\\
    &= n\frac{\psi_{\kappa'}(x)}{\psi_{\kappa'}(n)} + ne^{-rn} \overset{n\rightarrow\infty}{\longrightarrow} 0,
\end{split}
\end{align}
where the limit follow from Lemma \ref{Lemma_that_was_previously_an_assumption}(III); to see that this lemma is applicable notice that $y\mapsto\mu_{\kappa'}(y)-(r-c)y$ is strictly decreasing on $[\kappa, \infty)$ (this is verified using Assumption \ref{Assumption_2} and the second line in \eqref{eq:sde-g=1-no-dev}). 
    
Lastly, we note that \eqref{eq:generalized-psi=expec-bound} implies that \eqref{xi_needs_to_die} holds, and we have thus proved the second claim. \hfill \qedsymbol{}

\subsection{Proof of Theorem \ref{Theorem_existence_g_NE_1}}\label{subsec:proof_g_existence}
We will use that the assumption that $g$ is $b'$-admissible implies that $g$ is $b$-admissible for any  $b\geq b'$.

\begin{enumerate}[(I)]
    \item Let us first consider $\underline{b}$ so that $\underline{b}>b',y_1, \kappa, b_2$ (where $b'$ is from the theorem statement, $y_1$ is from Definition \ref{def:Admissible-g}, $\kappa$ is from Assumption \ref{Assumption_2}, $b_2$ is from Lemma \ref{Lemma_that_I_don't_know_what_to_call}). We remark that $\underline{b}$ will be further specified below. 

    For $x\geq b$, the condition in \eqref{NE:cond2-g} is directly verified, for any $b\geq 0$. Using Lemma \ref{Lemma_that_I_don't_know_what_to_call}(I)-(III) we find that that $\psi'(x) \geq \psi'(b)>0$ for $x\leq b$ for any sufficiently large $b$. Using also \eqref{eq:newVb} we find, for $x<b$ and any fixed sufficiently large $b$, that $V_b'(x)=g(b)\frac{\psi'(x)}{\psi'(b)}\geq g(b)=g(y_1) \geq g(x)$ (the last equality and the last inequality hold by assumption). Hence, \eqref{NE:cond2-g} holds, for any $b \geq \underline{b}$ for a sufficiently large $\underline{b}$.
        
    Let us now show that \eqref{NE:cond1-g} holds for any $b\geq \underline{b}$.
    Using Lemma \ref{Lemma_that_I_don't_know_what_to_call}(I)-(II) we find, for $b\geq b_2$, that:
    \begin{equation*}
        0\geq \frac{\sigma^2(b)}{2}\psi''(b)=r\psi(b)-\mu(b)\psi'(b)\implies
        \mu(b)-r\frac{\psi(b)}{\psi'(b)}\geq 0,
    \end{equation*}
    (where we also used that $\psi$ solves \eqref{eq:ode-for-psi}). Since $\underline{b}>b_2,y_1,\kappa$, for any $x$ and $b$ with $x\geq b \geq \underline{b}$ we then (remembering the constant $c$ from Assumption \ref{Assumption_2} and that $g$ is constant on $[\underline{b},\infty)$) have that the numerator in \eqref{eq:lambda-inthmwithg} satisfies, for $x \in (b, \infty)\setminus\Theta$: 
    \begin{align*}
        A_X V_b(x)-r V_b(x) &=
        \frac{\sigma^2(x)}{2}g'(x)+\mu(x)g(x)-rg(b)\frac{\psi(b)}{\psi'(b)}-r(G(x)-G(b))\\
        &=\mu(x)g(b)-rg(b)\frac{\psi(b)}{\psi'(b)}-r(x-b)g(b)\\
        &\geq(\mu(b)+(r+c)(x-b))g(b)-rg(b)\frac{\psi(b)}{\psi'(b)}-r(x-b)g(b)\\
        &=\left(\mu(b)-r\frac{\psi(b)}{\psi'(b)}\right)g(b)+c(x-b)g(b)\\
        &\geq\left(\mu(b)-r\frac{\psi(b)}{\psi'(b)}\right)g(b) \geq 0.
    \end{align*}
    Hence, \eqref{NE:cond1-g} holds for any $b\geq \underline{b}$ for a sufficiently large $\underline{b}$. The result follows.
    \item For $x\geq b$, the condition in \eqref{NE:cond2-g} is directly verified, for any $b\geq 0$. Lemma \ref{Cor:concave-psi} implies that $\psi$ is concave on $[0, b)$. 
    Hence, $V_{b}(x) =  \frac{\psi'(x)}{\psi'(b)}g(b) \geq g(b) \geq g(x),x < b$ (where the last inequality holds by the assumptions in the theorem statement, and the assumption that $1\geq g(x)$, cf. Definition \ref{def:Admissible-g}). The result follows.
    \item From statement (II) in the present theorem we have that \eqref{NE:cond2-g} holds. Using the assumptions in the theorem statement we have, for any $b\geq b'$, that the numerator in \eqref{eq:lambda-inthmwithg} satisfies, for $x \in (b, \infty)\setminus\Theta$ (note $g'(x)\geq0$ on this set in the present case):
    \begin{align*}
        A_XV_b(x)-rV_b(x) &= \frac{\sigma^2(x)}{2}g'(x)+
        \mu(x)g(x)-rg(b)\frac{\psi(b)}{\psi'(b)} -r(G(x)-G(b))\\
        &\geq (\mu(b)+r(x-b))g(x)-rg(b)\frac{\psi(b)}{\psi'(b)}-rg(x)(x-b)\\
        &=\mu(b)g(x)-rg(b)\frac{\psi(b)}{\psi'(b)}\geq 
        g(b)\left(\mu(b)-r\frac{\psi(b)}{\psi'(b)}\right).
    \end{align*}
    Lemma \ref{Cor:concave-psi} implies that $\psi$ concave on $[0,b)$ (for any $b\geq0$). Using also Lemma \ref{Lemma_that_I_don't_know_what_to_call}(I) we find:
    \begin{equation*}
        \mu(b) - r\frac{\psi(b)}{\psi'(b)}=-\frac{\sigma^2(b)}{2}\frac{\psi''(b)}{\psi'(b)}\geq0.
    \end{equation*}
    The observations above imply that \eqref{NE:cond1-g} holds. The result follows.
\end{enumerate}
\hfill \qedsymbol{}

\subsection{Examples}\label{subsec:examples_g}
Here we use the developed theory to study two examples. The first example is an extension of Example \ref{first-example}.
\begin{example}\label{Example:g_jump} 
    Here we consider a version of Example \ref{first-example} and search for an equilibrium of the kind in Theorem \ref{Theorem_verification_g_NE_2}. Suppose $\sigma(x)\equiv\sigma>0$ and $\mu(x) = \mu x$ where $\mu>r$ is a constant, then Assumptions \ref{Assumption_1} and \ref{Assumption_2} are immediately verified. Consider:
    \begin{equation*}
        g(x) = \begin{cases}
            a_{1} & \text{when: }  0 \leq x<\ell\\
            a_{1}+a_{2} & \text{when: }x\geq\ell,
        \end{cases}
    \end{equation*}
    where $a_i>0$ and $\ell >0$  are constants (see Remark \ref{remark:motivating-example} for a specific case and an interpretation). In this case the condition for the skew point  intensities in Definition \ref{def:Admissible-g} (i.e.,  $c_j\in (0,\frac{1}{2}]$, with $c_j$ defined in \eqref{eq:def-c_j}) simplifies to $a_{2}\leq 2 a_{1}$; therefore we assume that $a_1$ and $a_2$ satisfy this condition. Under this condition we have, by Theorem \ref{Theorem_existence_g_NE_1}(III) (the conditions of which are easily verified), a global Markovian Nash equilibrium for each $b\geq0$; see Figure \ref{fig:example2} for illustrations. 
    \begin{figure}[ht]
        \centering
        \includegraphics[width=0.45\linewidth]{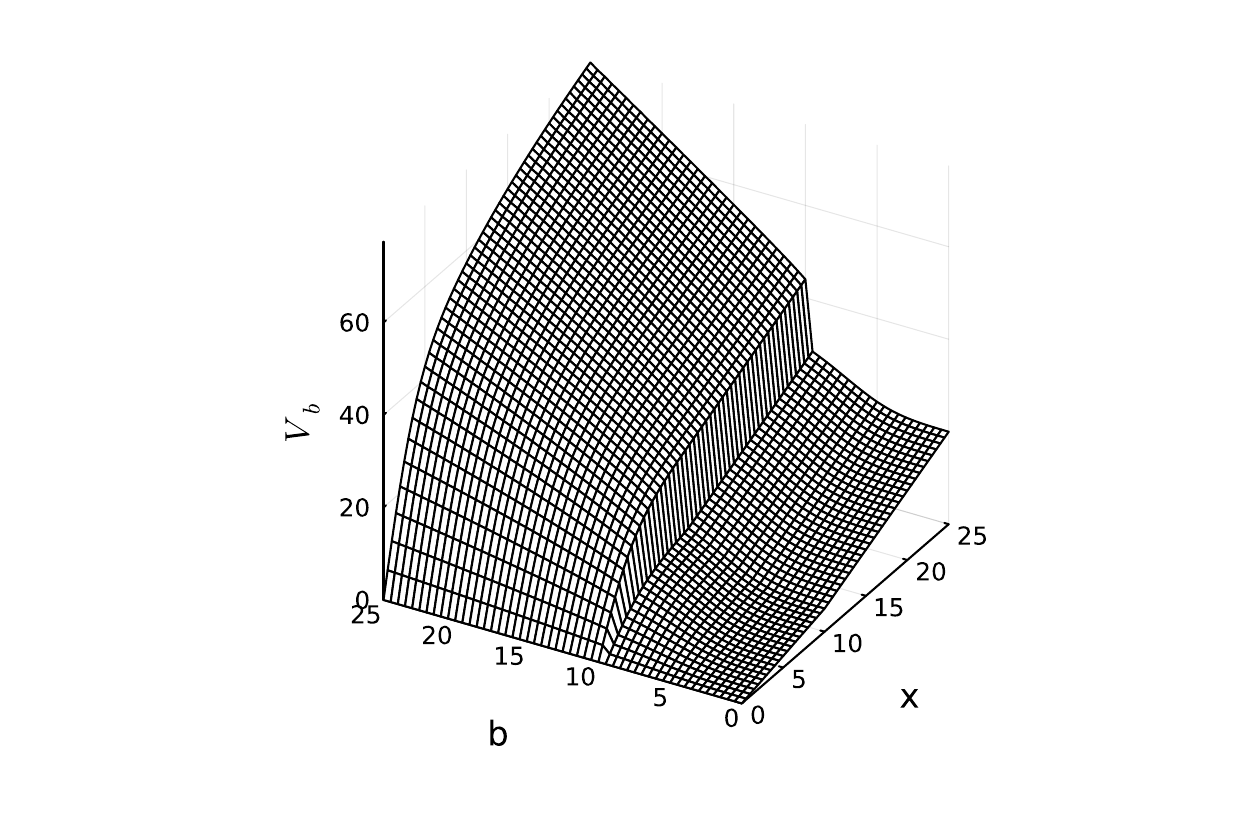}
        \includegraphics[width=0.45\linewidth]{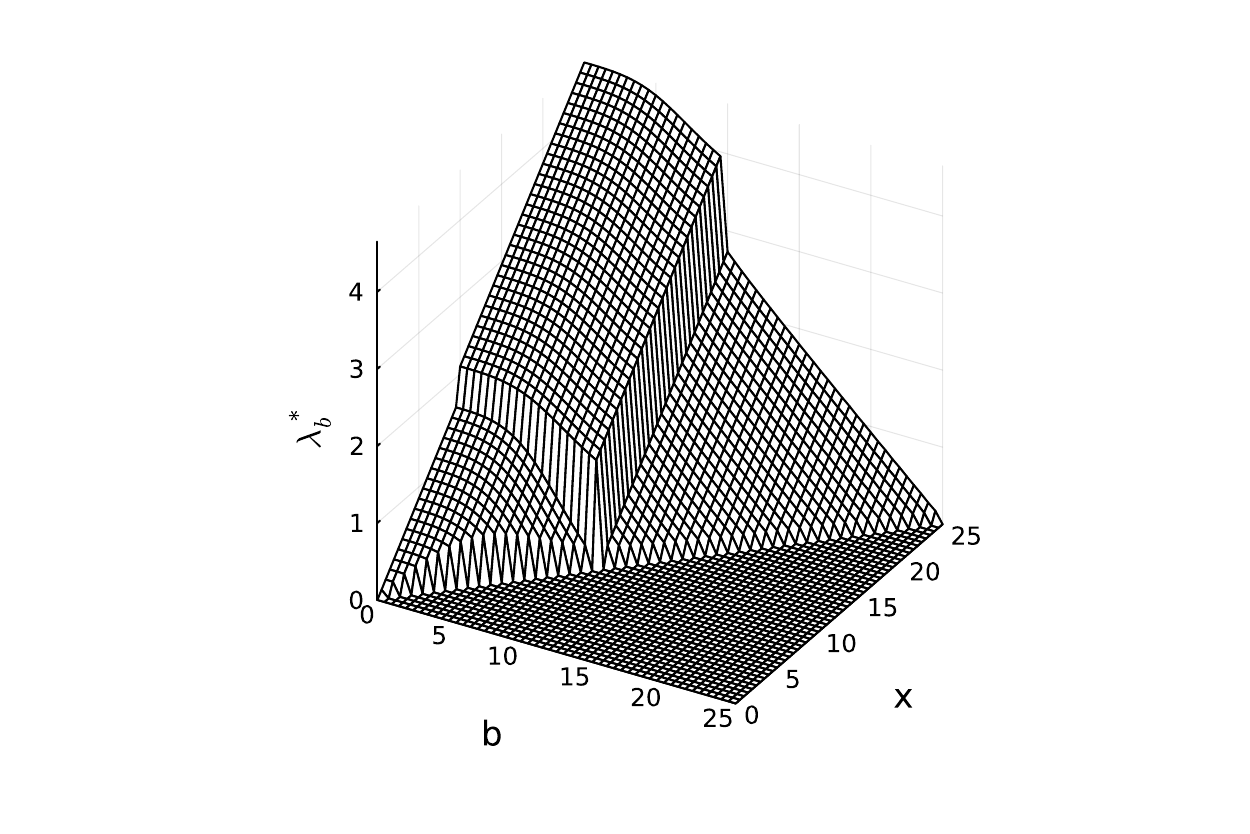}
        \caption{Illustrations for Example \ref{Example:g_jump} (based on standard numerical methods, cf. Example \ref{first-example}) for $r=0.08$, $\mu=0.25$, $\sigma=2$, $\ell = 10$, $a_1 = a_2 = \frac{1}{2}$ (this corresponds to the $g$ in Remark \ref{remark:motivating-example}). 
        The value function $V_b$ (given by \eqref{eq:newVb}) is increasing in $x$ and in $b$, and there is a discontinuity in $b\mapsto V_b$ for $b=\ell$ (where the profit rate $g$ has its discontinuity). The extraction rate $\lambda_b^\ast$ is increasing in $x$ and decreasing in $b$ with a discontinuity at $x=b$ due to the indicator function (cf. \eqref{eq:lambda-inthmwithg}), while the other  discontinuities appear due to discontinuities in $x \mapsto V_b'(x)$ and $b \mapsto  V_b(x)$. For $b < \ell$ the extraction equilibrium  strategy includes a skew point (not illustrated here) with intensity $c_1=\frac{1}{3}$ at $x=\ell$ (calculated using \eqref{eq:def-c_j}), while there is no  skew point in case $b\geq\ell$.}
        \label{fig:example2}
    \end{figure}
\end{example}

\begin{example}\label{Example:g_complicated}
    Here we consider an example with a profit rate function $g$ that exhibits continuos variability as well as discontinuities and search for an equilibrium of the kind in Theorem \ref{Theorem_verification_g_NE_2}. The example illustrates how the control rate and the skew points of the equilibrium strategy interact (see Figure \ref{fig:3}) and how the diffusion coefficient $\sigma$ can influence the equilibrium control rate (see \eqref{eq:ex2}). Consider:  
    \begin{equation*}
        r=\frac{1}{10}, \enskip \mu(x) = \frac{11}{10}rx + \frac{1}{1000}, \enskip\sigma(x) = 2\frac{1+x}{2+x},\enskip g(x) = \begin{cases}
            \frac{1}{17} & \text{when: }  0 \leq x<1\\
            \frac{1}{6} & \text{when: }1\leq x<3\\
            \frac{1}{6} + \frac{x^2 - 6x + 9}{24} & \text{when: }3\leq x<5\\
            1 & \text{when: }x\geq 5.
        \end{cases}
    \end{equation*}
    Assumptions \ref{Assumption_1} and \ref{Assumption_2} are immediately verified. It is easily verified that we have a global Markovian Nash equilibrium for any $b\geq 0$ (use Theorem \ref{Theorem_existence_g_NE_1}(III)). In this example we consider the equilibrium for $b=0$, which implies that the value function is given by $V_b=G$. The associated equilibrium strategy corresponds to using skew points at 
    $x \in \{1, 5\}$ with intensities $\frac{11}{23}$ and $\frac{1}{2}$ (calculated according to \eqref{eq:def-c_j}), and the equilibrium rate is given by:
    \begin{align}\label{eq:ex2}
    \begin{split}
        \lambda^{\ast}(x) &= 
        \left(\frac{\sigma^2(x)}{2}\frac{g'(x)}{g(x)} + \mu(x) - r \frac{G(x)}{g(x)}\right)\1_{\{x\notin\{1, 5\}\}} \\
        &= \begin{cases}
            0 & \text{when: }x\in\{1, 5\}\\
            \frac{1}{100}x + \frac{1}{1000} & \text{when: } 0 \leq x<1\\
            \frac{1}{100}x + \frac{1117}{17000} & \text{when: } 1<x<3 \\
            2\left(\frac{1+x}{2+x}\right)^2\frac{2x-6}{x^2 -6x + 13} + 
            \frac{11}{100}x + \frac{1}{1000}-
            \frac{\frac{1}{17} + \frac{2}{6} + 
            \frac{x^3 - 9  x^2 + 39  x - 63}{72}}{10\frac{x^2-6x+13}{24}} & 
            \text{when: }3\leq x<5 \\
            \frac{11}{100}x + \frac{1}{1000} - \frac{x-5 + \frac{128}{153}}{10} & 
            \text{when: } x>5.
        \end{cases}
        \end{split}
    \end{align}
    \begin{figure}[ht]
        \centering
        \includegraphics[width=0.45\linewidth]{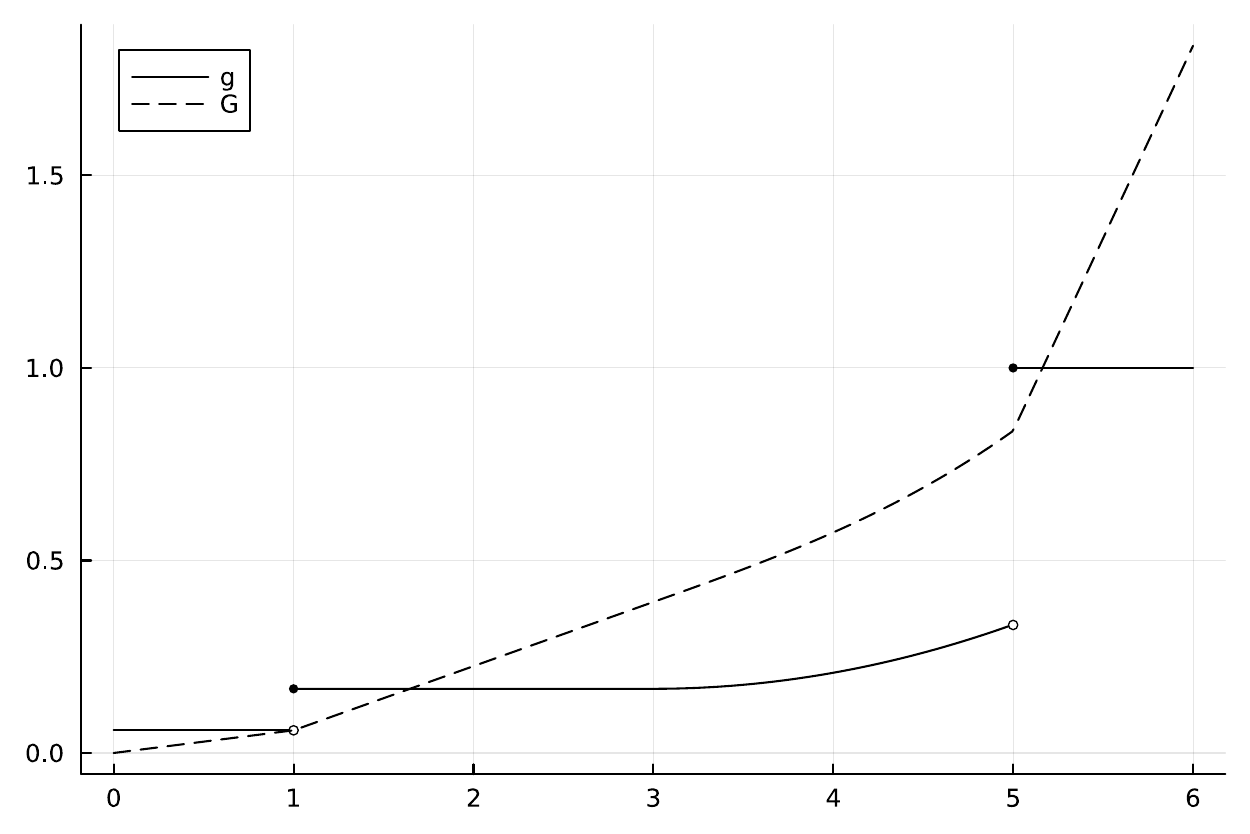}
        \includegraphics[width=0.45\linewidth]{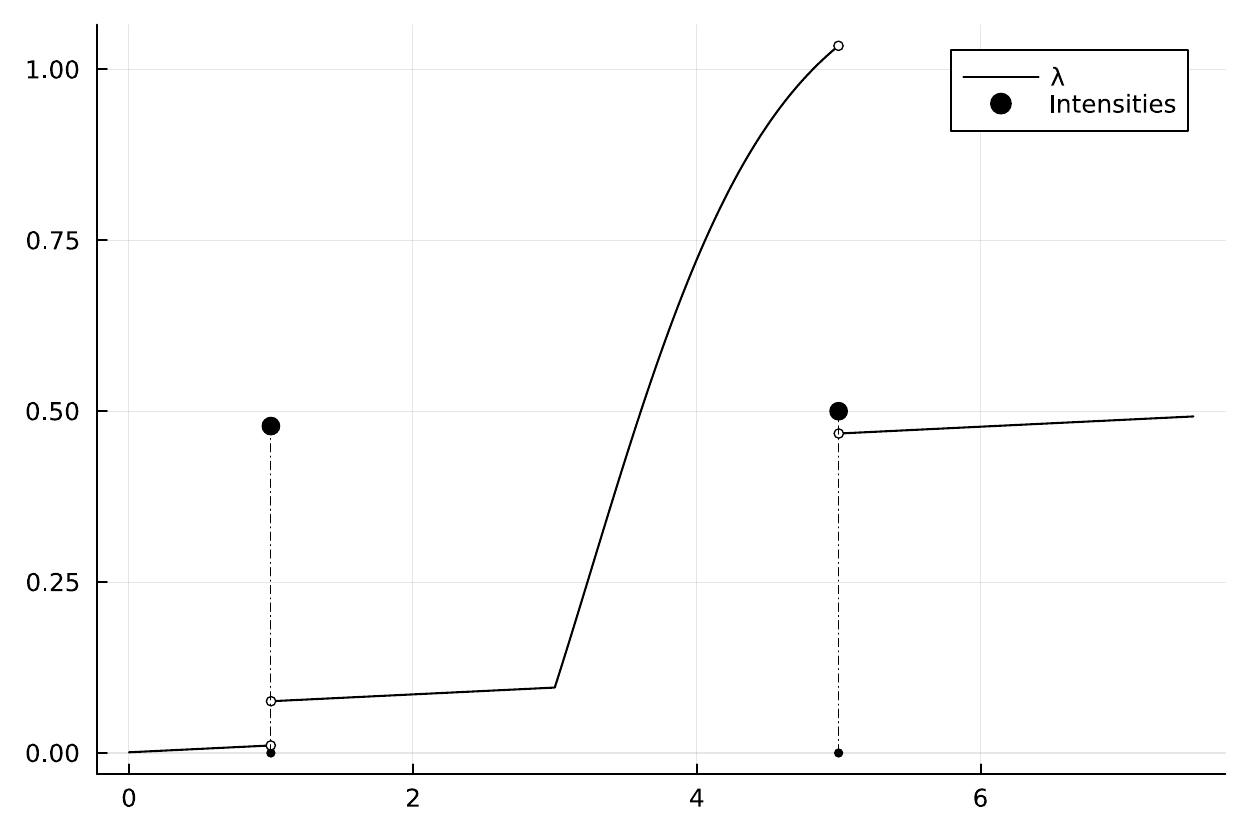}
        \caption{Illustrations for Example \ref{Example:g_complicated}. The value function $V_b=G$  (cf. $b=0$), and the profit rate function $g$ are illustrated in the first panel. The second panel illustrates the skew points and their intensities as well as the extraction rate of the equilibrium strategy.}
         \label{fig:3}
    \end{figure}
\end{example}

\section{A case study for equilibria with jumps and reflection}\label{sec:solution-with-jumps}
A careful reading of the previous section reveals that the theory developed there does not allow us to find an equilibrium in which the players control with jumps (to see this note that the jump region of the equilibrium strategy in Theorem \ref{Theorem_verification_g_NE_2} is the empty set). Hence, it is natural to ask:  can we find equilibria with jumps? 

The answer to this question is yes, and to show that this is so we will here consider a case study in which we find an equilibrium such that the players' collective control adds up to the usual (in the context of singular stochastic control theory) bang-bang type strategy of reflecting at a given boundary and jumping to that boundary if the process starts above it. 

In particular, we consider in this section a profit rate that is piecewise constant with one jump point $\ell>0$ that earns a higher profit rate, such that above this jump point the profit rate is zero:
\begin{equation}\label{eq:g-case-study}
    g(x) = \begin{cases}
        w & \text{for: }x<\ell,\\
        1 & \text{for: }x=\ell\\
        0 & \text{for: }x>\ell,  
        \enskip \text{for fixed $\ell>0$ and }w\in\left(0, \frac{1}{3}\right].
    \end{cases}
\end{equation}
(Note that the definition in \eqref{g-function-payouts2} can be trivially extended to allow for functions $g$ that are not RCLL, including $g$ given by \eqref{eq:g-case-study}, which we rely on below. Indeed, this definition is valid for e.g., non-negative functions that are locally integrable and have left limits.) 

\begin{theorem}[Equilibrium existence and characterization]\label{Example-with-jumps}
    Suppose Assumption \ref{Assumption_1} holds. Consider the game (Definition \ref{def:Markov-NE-for-g}) corresponding to $g$ defined in \eqref{eq:g-case-study}. Suppose $\mu(0)\geq0$ and that $x\mapsto \mu(x)-rx$ is non-decreasing on $[0, \ell)$. Then the pair of strategies $(D^1, D^2)$, where:
    \begin{equation*}
        D^i:= \left\{0, \{\ell\}, \left\{\frac{1}{2}\right\},
        x\mapsto \frac{\1_{\{x>\ell\}}}{2}(x-\ell), 
        [\ell, \infty)\right\},
    \end{equation*} 
    constitutes a global Markovian Nash equilibrium with a value function $V_{\ell}$ given by:
    \begin{equation}\label{V_b:case-study}
        V_{\ell}(x) = \begin{cases}
            \frac{g(\ell) + g(\ell-)}{4}\frac{\psi(x)}{\psi'(\ell)} & 
            \text{when: } 0 \leq x< \ell\\
            \frac{g(\ell) + g(\ell-)}{4}\frac{\psi(\ell)}{\psi'(\ell)} & 
            \text{when: }x\geq \ell.
        \end{cases}
    \end{equation}
\end{theorem}
The proof of Theorem \ref{Example-with-jumps} is similar to the proof of Theorem \ref{Theorem_verification_g_NE_2} and can be found in Appendix \ref{thmproofappendix}.

\begin{remark} 
    The equilibrium in Theorem \ref{Example-with-jumps} is of the bang-bang type mentioned in the beginning of this section. In particular, the equilibrium strategy can be described as follows: 
    \begin{itemize}
        \item The rate control is zero for each player. 
        \item Each player controls  with a local time push with intensity $\frac{1}{2}$ at $\ell$, implying that the aggregated local time push intensity at $\ell$ is $1$ so that the controlled state processes is reflected at $\ell$ when reaching $\ell$ from below (see the proof of Theorem \ref{Example-with-jumps} for details).  
        \item Each player controls with an initial jump $\frac{x-\ell}{2}$ in the state process whenever the state process starts in  $(\ell, \infty)$, so that there is an initial jump to $\ell$ in this case.
        \item The intuition for \eqref{V_b:case-study} is thereby clear; in particular, given that the state process is controlled only so that it is reflected at $\ell$, we find, using \eqref{g-function-payouts}--\eqref{g-function-payouts2}, that the corresponding values satisfy (for $x\leq\ell$):
        \begin{align*}
            J_{i}\left(x; D^1,D^2\right) 
            &= \E_{x}\left[\int_{[0,\tau]} 
            e^{-rt}g\left(X^{D^1,D^2}_{t-}\right)\circ d D_{t}^{i}\right]\\
            &= \E_x\left[\int_{[0,\tau]}e^{-rt}  
            \frac{g(\ell) + g(\ell-)}{2} \frac{1}{2}dL^{\ell}_t\right]\\
            &= \E_x\left[\int_{[0,\tau]}
            e^{-rt}dL^{\ell}_t\right]\frac{g(\ell) + g(\ell-)}{4}=V_{\ell}(x).
        \end{align*}
        where the last equality is a standard result (see the proof for details).
\end{itemize}
\end{remark}

%% file: appendix.tex
\appendix
\section{Results for increasing fundamental solutions}\label{ODEappendix}
This appendix includes results for the increasing fundamental solution $\psi$ (defined in the beginning of Section \ref{sec:Naive_Lebesgue_absolutely_continuous_Nash_equilibria}) as well as another increasing fundamental solution $\hat \psi$ (see \eqref{hatODE}) that we rely on in the present paper. We note that similar results are often used in the study of stochastic control problems of optimal dividend type, although the present results require weaker assumptions than those usually presented (cf. e.g., \cite[Sec. 7]{bai2012optimal}, \cite[Lemma 2.4]{ekstrom2023finetti} and \cite[Sec. 2 and App A]{christensen2022moment}); in particular, we make no assumption regarding continuos differentiability of $\mu$ and $\sigma$, and we assume a relationship between the growth of $x\mapsto \mu(x)$ and $x\mapsto rx$ only on a subset of the real line.
\begin{lemma}\label{Lemma_that_was_previously_an_assumption}
    Consider $\widehat{\mu}:[0, \infty)\rightarrow\mathbb{R}$ which is assumed to be globally Lipschitz continuous and $\sigma:[0, \infty)\rightarrow(0, \infty)$ which we suppose satisfies Assumption \ref{Assumption_1} (i.e., $\sigma$ is also globally Lipschitz continuous). Suppose $\displaystyle x\mapsto \widehat{\mu}(x)-(r-c)x$ is strictly decreasing on $[\kappa, \infty)$, given constants $\kappa>0$ and $c>0$. Let $\widehat{\psi}$ be the fundamental increasing $\mathcal{C}^{2}\left[0,\infty\right)$-solution to:
    \begin{equation}\label{hatODE}
    \frac{\sigma^{2}(x)}{2}\widehat{\psi}''(x)+ \widehat{\mu}(x)\widehat{\psi}'(x)=r\widehat{\psi}(x), \enskip \widehat{\psi}'(0)=1 \enskip, \widehat{\psi}(0)=0.
    \end{equation}
    \begin{enumerate}[(I)]
        \item $\widehat{\psi}(x)>0$, for $x\geq0$.    
        \item There exists a unique point $b_1\in [\kappa,\infty)$ such that $\widehat{\psi}''(x)<0, x\in (\kappa,b_{1})$, $\widehat{\psi}''(x) >0, x\in (b_1,\infty)$; in particular, $\widehat{\psi}$ is convex or concave-convex when restricted to  $[\kappa,\infty)$.  
        \item \begin{align}\label{Lemma_that_was_previously_an_assumption:eq}
            \frac{n}{\widehat{\psi}(n)}\overset{n\rightarrow\infty}{\longrightarrow}0.
        \end{align}
    \end{enumerate}
\end{lemma}
\begin{proof}
\begin{enumerate}[(I)]
    \item Note that if we would have $\widehat{\psi}'(x) = 0$ for some $x>0$, then the property $\widehat{\psi}'(x) \geq 0, x \geq 0$  would imply that $x$ is a local minimum of $\widehat{\psi}'$, implying in turn that $\widehat{\psi}''(x) = 0$, so that the ODE for $\widehat{\psi}$ (i.e., \eqref{hatODE}) and $\widehat{\psi}'(0)=1$, would yield the contradiction:
    \begin{equation*}
        0 = \frac{\sigma^{2}(x)}{2}\widehat{\psi}''(x) =
        -\widehat{\mu}(x)\widehat{\psi}'(x) + r\widehat{\psi}(x) = 
        r\widehat{\psi}(x)>0.
    \end{equation*}
    \item The proof consists in two steps: \textbf{1.} First we show that it cannot be the case that $\widehat{\psi}$ is strictly concave on $[\kappa, \infty)$, i.e., we show that we do not have $\widehat{\psi}''(x)<0,x \in [\kappa, \infty)$. We use a contradiction argument, in particular, suppose $\widehat{\psi}''(x)<0,x \in [\kappa, \infty)$, relying on our assumption for $\widehat{\mu}$ we select a constant $x_{1}\geq\kappa$ with $\widehat{\mu}(x_{1})<\left(r-\frac{c}{2}\right)x_{1}$, which implies that $\widehat{\mu}(x_{2})<\left(r-\frac{c}{2}\right)x_{2}$ for any $x_{2}\geq x_{1}$. Now use the contradiction assumption, the ODE and $\widehat{\psi}'(x)>0$ (i.e., (I)) to find that we for any $x_{2}>x_{1}$ have:
    \begin{align*}
        0>\frac{\sigma^{2}(x_{2})}{2}\widehat{\psi}''(x_{2}) &=
        r\widehat{\psi}(x_{2}) - \widehat{\mu}(x_{2})\widehat{\psi}'(x_{2})\\
        &\geq r\widehat{\psi}(x_{2}) - \max\left(0, \widehat{\mu}(x_{2})\right)\widehat{\psi}'(x_{2})\\
        &\geq r\widehat{\psi}(x_{2}) - \max\left(0, \widehat{\mu}(x_{2})\right)
        \frac{\widehat{\psi}(x_{2}) - \widehat{\psi}(x_{1})}{x_{2} - x_{1}}\\
        &\geq r\widehat{\psi}(x_{2}) - 
        \max\left(0, \left(r-\frac{c}{2}\right)x_{2}\right)
        \frac{\widehat{\psi}(x_{2}) - \widehat{\psi}(x_{1})}{x_{2} - x_{1}}.
    \end{align*}
    Multiplying both sides of the inequality by $(x_{2} - x_{1})$ gives, for any $x_{2}>x_{1}$:
    \begin{align*}
        0 &> r\widehat{\psi}(x_{2})\left(x_{2} - x_{1}\right)-
        \max\left(0, r-\frac{c}{2}\right)x_{2} \left(\widehat{\psi}(x_{2}) - \widehat{\psi}(x_{1})\right)\\
        &\geq r\widehat{\psi}(x_{2}) x_{2} - x_{1}r\widehat{\psi}(x_{2}) -\max\left(0, r-\frac{c}{2}\right)x_{2}\widehat{\psi}(x_{2})\\
        &= \widehat{\psi}(x_{2}) x_{2}\left(r - \max\left(0, r-\frac{c}{2}\right)\right)- x_{1}r\widehat{\psi}(x_{2})\\
        &= \widehat{\psi}(x_{2})\left(x_{2}\min\left(r,\frac{c}{2}\right) - x_{1}r\right).
    \end{align*}
    However, it is clear that the last line of the above is positive for sufficiently large $x_{2}$, and we have thus arrived at a contradiction.

    \textbf{2.} By \textbf{1.} above we have two cases: either $\widehat{\psi}''>0$ on $[\kappa,\infty)$ (in which case the statement of (II) holds with $b_1=\kappa$), or there exists at least one point $z\in [\kappa,\infty)$ with $\widehat{\psi}''(z)=0$. Using the conclusions of the following items (and continuity of $\psi''$) it is easy to see that statement in (II) holds also in the latter case (which thus concludes the proof):
    \begin{itemize}
        \item Using our assumption for $\widehat{\mu}$ and (I) above we find, for any $\varepsilon>0$, that
        \begin{align}\label{eq:aweqwre}
        \begin{split}
            \widehat{\mu}(z+\varepsilon)&<
            \widehat{\mu}(z)+(r-c)\varepsilon\\
            \implies\widehat{\mu}(z+\varepsilon)
            \widehat{\psi}'(z+\varepsilon)
            &<\widehat{\mu}(z)\widehat{\psi}'(z+\varepsilon)+
            \varepsilon(r-c)\widehat{\psi}'(z+\varepsilon).
        \end{split}
        \end{align}
        Using \eqref{eq:aweqwre}, the ODE and $\widehat{\psi}''(z)=0$ we obtain (in the last equality below we use that $r\widehat{\psi}(z) - \widehat{\mu}(z)\widehat{\psi}'(z)=0$ for our $z$):
        \begin{align*}
            \frac{\sigma^{2}(z+\varepsilon)}{2}\widehat{\psi}''(z+\varepsilon)&=
            r\widehat{\psi}(z+\varepsilon)- \widehat{\mu}(z+\varepsilon)
            \widehat{\psi}'(z+\varepsilon)\\
            &=r\widehat{\psi}(z) + r\int_{z}^{z+\varepsilon}\widehat{\psi}'(t)dt-
            \widehat{\mu}(z+\varepsilon)\widehat{\psi}'(z+\varepsilon)\\
            &>r\widehat{\psi}(z) + r\int_{z}^{z+\varepsilon}\widehat{\psi}'(t)dt-
            \widehat{\mu}(z)\widehat{\psi}'(z+\varepsilon)-
            \varepsilon(r-c)\widehat{\psi}'(z+\varepsilon)\\
            &= r\widehat{\psi}(z) - \widehat{\mu}(z)\widehat{\psi}'(z) + 
            r\int_{z}^{z+\varepsilon}\widehat{\psi}'(t)dt\\
            &-\widehat{\mu}(z)\left(\widehat{\psi}'(z+\varepsilon)-\widehat{\psi}'(z)\right)-\varepsilon(r-c)\widehat{\psi}'(z+\varepsilon)\\
            &= r\int_{z}^{z+\varepsilon}\widehat{\psi}'(t)dt
            -\widehat{\mu}(z)\left(\widehat{\psi}'(z+\varepsilon)-
            \widehat{\psi}'(z)\right)-\varepsilon(r-c)\widehat{\psi}'(z+\varepsilon).
        \end{align*}
        Hence, we find:
        \begin{align*}
            \frac{\sigma^{2}(z+\varepsilon)}{2\varepsilon}\widehat{\psi}''(z+\varepsilon) &>
            \frac{r}{\varepsilon}\int_{z}^{z+\varepsilon}\widehat{\psi}'(t)dt
            -\widehat{\mu}(z)\frac{ \widehat{\psi}'(z+\varepsilon)-
            \widehat{\psi}'(z)}{\varepsilon}-
            (r-c)\widehat{\psi}'(z+\varepsilon)\\
            &\overset{\varepsilon\searrow0}{\longrightarrow}
            r\widehat{\psi}'(z)-\ \widehat{\mu}(z)
            \widehat{\psi}''(z)-
            (r-c)\widehat{\psi}'(z)=c\widehat{\psi}'(z)>0.
        \end{align*}
        Since $\sigma>0$ (Assumption \ref{Assumption_1}) we thus conclude that $\widehat{\psi}''(z+\varepsilon)>0$ for any sufficiently small $\varepsilon>0$.
        \item Consider the case $z>\kappa$. Using arguments similar to those above we then have, for any $\varepsilon \in (0,z-\kappa)$:
        \begin{align*}
            \widehat{\mu}(z-\varepsilon)&>\widehat{\mu}(z)+(c-r)\varepsilon\\
            \implies\widehat{\mu}(z-\varepsilon)\widehat{\psi}'(z-\varepsilon)
            &>\widehat{\mu}(z)\widehat{\psi}'(z-\varepsilon)+
            \varepsilon(c-r)\widehat{\psi}'(z-\varepsilon),
        \end{align*}
        and:
        \begin{align*}
            \frac{\sigma^{2}(z-\varepsilon)}{2}\widehat{\psi}''(z-\varepsilon)&=
            r\widehat{\psi}(z-\varepsilon)-
            \widehat{\mu}(z-\varepsilon)\widehat{\psi}'(z-\varepsilon)\\
            &= r\widehat{\psi}(z) - r\int_{z-\varepsilon}^{z}\widehat{\psi}'(t)dt-
            \widehat{\mu}(z-\varepsilon)\widehat{\psi}'(z-\varepsilon)\\
            &< r\widehat{\psi}(z) - r\int_{z-\varepsilon}^{z}\widehat{\psi}'(t)dt-
            \widehat{\mu}(z)\widehat{\psi}'(z-\varepsilon)-
            \varepsilon(c-r)\widehat{\psi}'(z-\varepsilon)\\
            &=r\widehat{\psi}(z) - r\int_{z-\varepsilon}^{z}\widehat{\psi}'(t)dt-
            \widehat{\mu}(z)\widehat{\psi}'(z)-
            \widehat{\mu}(z)\left(\widehat{\psi}'(z-\varepsilon)-\widehat{\psi}'(z)\right)-
            \varepsilon(c-r)\widehat{\psi}'(z-\varepsilon)\\
            &=-r\int_{z-\varepsilon}^{z}\widehat{\psi}'(t)dt -
            \widehat{\mu}(z)\left(\widehat{\psi}'(z-\varepsilon)-\widehat{\psi}'(z)\right)-
            (c - r)\varepsilon\widehat{\psi}'(z-\varepsilon),
        \end{align*}
        so that:
        \begin{align*}
            \frac{\sigma^{2}(z-\varepsilon)}{2\varepsilon}\widehat{\psi}''(z-\varepsilon) &
            <
            -\widehat{\mu}(z)\frac{\widehat{\psi}'(z-\varepsilon) - 
            \widehat{\psi}'(z)}{\varepsilon} -
            \frac{r}{\varepsilon}\int_{z-\varepsilon}^{z}\widehat{\psi}'(t)dt -
            (c - r)\widehat{\psi}'(z-\varepsilon)\\
            &\overset{\varepsilon\searrow0}{\longrightarrow}
            \widehat{\mu}(z)\widehat{\psi}''(z)
            -r\widehat{\psi}'(z)-(c - r)\widehat{\psi}'(z) = -c\widehat{\psi}'(z)<0.
        \end{align*}
        We conclude that $\widehat{\psi}''(z-\varepsilon)<0$ for all sufficiently small $\varepsilon>0$, in case $z>\kappa$.
    \end{itemize}
    \item From (I) and (II) we know that $\widehat{\psi}$ is increasing and ultimately convex ($\widehat{\psi}''(x)>0$ for all sufficiently large $x$), so that:
    \begin{equation*}
        \lim_{x\rightarrow\infty}\frac{\widehat{\psi}(x)}{x}=
        \lim_{x\rightarrow\infty}\widehat{\psi}'(x)=C,
    \end{equation*}
    for some $C\in(0, \infty]$. Let us now show that $C=\infty$, so that the above equality gives \eqref{Lemma_that_was_previously_an_assumption:eq}, which thus concludes the proof. Assume towards contradiction that $C\in(0, \infty)$, and use our assumption for $\widehat{\mu}$ and the ODE for $\widehat{\psi}$ to see that there exists a constant $A$ such that, for all $x>\kappa$, we have
    \begin{align*}
        \frac{\sigma^{2}(x)}{x}\widehat{\psi}''(x)&=
        r\frac{\widehat{\psi}(x)}{x}-\frac{\widehat{\mu}(x)}{x}\widehat{\psi}'(x)\\
        &\geq r\frac{\widehat{\psi}(x)}{x}-\frac{(r - c)x + A}{x}\widehat{\psi}'(x)
        \overset{x\rightarrow\infty}{\longrightarrow}
        rC -(r-c)C=cC>0.
    \end{align*}
    We also find that:
    \begin{equation*}
        \infty>C=\lim_{x\rightarrow\infty}\widehat{\psi}'(x) =
        \widehat{\psi}'(0) + \int_{0}^{\infty}\widehat{\psi}''(u)du,
    \end{equation*}
    which implies that there exists a sequence $(x_{k})_{k\in\mathbb{N}}$ increasing to infinity with:
    \begin{align*}
       \lim_{k\rightarrow\infty}\widehat{\psi}''(x_{k})x_{k}=0.
    \end{align*}
    Now use Assumption \ref{Assumption_1} to see that:
    \begin{align}\label{limit_sigmasq_over_xsq}
        \limsup_{y\rightarrow\infty}\frac{\sigma^{2}(x)}{x^{2}}<\infty.
    \end{align}
    Putting the above pieces together we find a contradiction:
    \begin{align*}
       0<cC\leq \lim_{k\rightarrow\infty}
        \frac{\sigma^{2}(x_{k})}{x^{2}_{k}}\widehat{\psi}''(x_{k})x_{k}=0.
    \end{align*}
\end{enumerate}
\end{proof}
\begin{lemma}\label{Lemma_that_I_don't_know_what_to_call}
    Suppose Assumptions \ref{Assumption_1} and \ref{Assumption_2} hold. Recall the function $\psi$ defined in the beginning of Section \ref{sec:Naive_Lebesgue_absolutely_continuous_Nash_equilibria}. 
    \begin{enumerate}[(I)]
        \item $\psi'(x)>0$, for $x\geq0$. 
        \item There exists a unique point $b_2\in [\kappa,\infty)$ such that $\psi''(x) > 0, x\in \left(\kappa,b_2\right)$, $\psi''(x) < 0, x\in \left(b_2,\infty\right)$; in particular, $\psi$ is concave or convex-concave when restricted to  $[\kappa,\infty)$. (Recall the constant $\kappa$ from Assumption \ref{Assumption_2}). 
        \item \begin{align}\label{Lemma_that_I_don't_know_what_to_call:eq}
            \psi'(n)\overset{n\rightarrow\infty}{\longrightarrow}0.
        \end{align}
    \end{enumerate}
\end{lemma}

\begin{proof}
We remark that the proof of the present result is similar to the proof of Lemma 
\ref{Lemma_that_was_previously_an_assumption}.
\begin{enumerate}[(I)]
    \item This claim can be proved using the arguments in the proof of Lemma  \ref{Lemma_that_was_previously_an_assumption}(I).
    \item Recall the constants $c$ and $\kappa$ from Assumption \ref{Assumption_2}. The proof consists of two steps: \textbf{1.} First we show that it cannot be the case that $\psi$ is strictly convex on $[\kappa, \infty)$, i.e., we show that we do not have $\psi''(x)>0,x \in [\kappa, \infty)$. We use a contradiction argument; in particular, suppose $\psi''(x)>0,x \in [\kappa, \infty)$. 
    Select a constant  $x_{1}\geq\kappa$ so that $\mu(x_{1})>\left(r+\frac{c}{2}\right)x_{1}$, which also implies that 
    $\mu(x_{2})>\left(r+\frac{c}{2}\right)x_{2}$ for any $x_{2}\geq x_{1}$
    (use Assumption \ref{Assumption_2}). Now use the contradiction assumption, the ODE for $\psi$ (i.e., \eqref{eq:ode-for-psi}), and $\psi'(x)>0$ (i.e., (I)) to find that we for any $x_{2}>x_{1}$ have:
    \begin{align*}
     0<\frac{\sigma^{2}(x_{2})}{2}\psi''(x_{2}) 
    &= r\psi(x_{2}) - \mu(x_{2})\psi'(x_{2})\\
    &\leq r\psi(x_{2}) - \mu(x_{2})\frac{\psi(x_{2}) - \psi(x_{1})}{x_{2} - x_{1}}\\
    &\leq r\psi(x_{2}) - \left(r+\frac{c}{2}\right)x_{2} \frac{\psi(x_{2}) - \psi(x_{1})}{x_{2} - x_{1}}.
    \end{align*}
    Multiplying both sides of the inequality by $(x_{2} - y_1)$ yields:
    \begin{align*}
         0&<r\psi(x_{2})(x_{2}-x_{1}) - 
         \left(r+\frac{c}{2}\right)x_{2} \left(\psi(x_{2}) - \psi(x_{1})\right) \\
         &= x_2\left(-\frac{c}{2}\psi(x_{2})+\frac{c}{2}\psi(x_{1})+r\psi(x_{1})\right)
         -rx_{1}\psi(x_{2}).
    \end{align*}
    However, for large $x_{2}$ we have that the last line above is negative (since $\displaystyle\lim_{x\rightarrow\infty}\psi(x)=\infty$, cf. e.g., \cite[p. 18-19]{borodin2012handbook}), which is a contradiction.

    \textbf{2.} By \textbf{1.} above we have two cases remaining: either $\psi''<0$ on $[\kappa,\infty)$ (in which case the statement in (II) holds with $b_2=\kappa$), or there exists at least one point $z\in [\kappa,\infty)$ with $\psi''(z)=0$. Using the conclusions of the following items (and continuity of $\psi''$) it is easy to see that the statement in (II) holds also in the latter case (which thus concludes the proof):
    \begin{itemize} 
        \item Using Assumption \ref{Assumption_2} and the result (I) of the present lemma, we find, for any $\varepsilon>0$, that
        \begin{align}\label{eq:aweqwre-ver2}
        \begin{split}
            \mu(z+\varepsilon)&>\mu(z)+(r+c)\varepsilon\\
            \implies\mu(z+\varepsilon)\psi'(z+\varepsilon)&>
            \mu(z)\psi'(z+\varepsilon)+\varepsilon(r+c)\psi'(z+\varepsilon).
        \end{split}
        \end{align}
        Using the above, the ODE and $\psi''(z)=0$ we obtain (in the last equality below we use also the ODE to see that $r\psi(z) - \mu(z)\psi'(z)=0$):
        \begin{align*}
            \frac{\sigma^{2}(z+\varepsilon)}{2}\psi''(z+\varepsilon)&=
            r\psi(z+\varepsilon)-\mu(z+\varepsilon)\psi'(z+\varepsilon)\\
            & = r\psi(z) + r\int_{z}^{z+\varepsilon}\psi'(t)dt-
            \mu(z+\varepsilon)\psi'(z+\varepsilon)\\
            &<r\psi(z) + r\int_{z}^{z+\varepsilon}\psi'(t)dt-
            \mu(z)\psi'(z+\varepsilon)-
            \varepsilon(r+c)\psi'(z+\varepsilon)\\
            &= r\psi(z) - \mu(z)\psi'(z) + 
            r\int_{z}^{z+\varepsilon}\psi'(t)dt-
            \mu(z)\left(\psi'(z+\varepsilon)-\psi'(z)\right)-
            \varepsilon(r+c)\psi'(z+\varepsilon)\\
            &= r\int_{z}^{z+\varepsilon}\psi'(t)dt
            -\mu(z)\left(\psi'(z+\varepsilon)-\psi'(z)\right)-
            \varepsilon(r+c)\psi'(z+\varepsilon).
        \end{align*}
      Hence, we find:
        \begin{align*}
            \frac{\sigma^{2}(z+\varepsilon)}{2\varepsilon}\psi''(z+\varepsilon) &<
            \frac{r}{\varepsilon}\int_{z}^{z+\varepsilon}\psi'(t)dt
            -\mu(z)
            \frac{\psi'(z+\varepsilon)-\psi'(z)}{\varepsilon}-
            (r+c)\psi'(z+\varepsilon)\\
            &\overset{\varepsilon\searrow0}{\longrightarrow}
            r\psi'(z)-\mu(z)\psi''(z)-(r+c)\psi'(z)=-c\psi'(z)<0.
        \end{align*}
        Since $\sigma>0$ we thus conclude that $\psi''(z+\varepsilon)<0$ for any sufficiently small $\varepsilon>0$.         
        \item Consider the case $z>\kappa$. Using arguments similar to those above we have, for any  $\varepsilon \in (0,z-\kappa)$:
        \begin{align}\label{eq:aweqwre_2-ver2}
        \begin{split}
            \mu(z-\varepsilon)+\varepsilon(r+c)&<\mu(z)\\
            \implies
            \mu(z-\varepsilon)\psi'(z-\varepsilon) &<
            \mu(z)\psi'(z-\varepsilon) - \varepsilon(r+c)\psi'(z-\varepsilon),
        \end{split}
        \end{align}
        and:
        \begin{align*}
            \frac{\sigma^{2}(z-\varepsilon)}{2}\psi''(z-\varepsilon)&=
            r\psi(z-\varepsilon)-\mu(z-\varepsilon)\psi'(z-\varepsilon)\\
            & = r\psi(z) - r\int_{z-\varepsilon}^{z}\psi'(t)dt-
            \mu(z-\varepsilon)\psi'(z-\varepsilon)\\
            &> r\psi(z) - r\int_{z-\varepsilon}^{z}\psi'(t)dt-
            \mu(z)\psi'(z-\varepsilon)+
            \varepsilon(c+r)\psi'(z-\varepsilon)\\
            &=r\psi(z) - r\int_{z-\varepsilon}^{z}\psi'(t)dt-
            \mu(z)\psi'(z)-\mu(z)\left(\psi'(z-\varepsilon)-\psi'(z)\right)+
            \varepsilon(c+r)\psi'(z-\varepsilon)\\
            &=  - r\int_{z-\varepsilon}^{z}\psi'(t)dt -
            \mu(z)\left(\psi'(z-\varepsilon)-\psi'(z)\right)+
            (c + r)\varepsilon\psi'(z-\varepsilon),
        \end{align*}
        so that:
        \begin{align*}
            \frac{\sigma^{2}(z-\varepsilon)}{2\varepsilon}\psi''(z-\varepsilon) &>
            -\mu(z)\frac{\psi'(z-\varepsilon) - \psi'(z)}{\varepsilon}-
            \frac{r}{\varepsilon}\int_{z-\varepsilon}^{z}\psi'(t)dt
            +(c + r)\psi'(z-\varepsilon)\\
            &\overset{\varepsilon\searrow0}{\longrightarrow}
            -r\psi'(z)+(c + r)\psi'(z) = c\psi'(z)>0.
        \end{align*}
        We conclude that $\psi''(z-\varepsilon)>0$ for all sufficiently small $\varepsilon>0$, in case $z>\kappa$.
    \end{itemize}
    \item From (I) and (II) we know that $\psi$ is increasing and ultimately concave ($\psi''(x)<0$ for all sufficiently large $x$), so that:
    \begin{equation*}
        \lim_{x\rightarrow\infty}\frac{\psi(x)}{x}=
        \lim_{x\rightarrow\infty}\psi'(x)=C,
    \end{equation*}
    for some $C\in[0, \infty)$. Let us now show that $C=0$, so that the above equality gives \eqref{Lemma_that_I_don't_know_what_to_call:eq}, which concludes the proof. Assume towards contradiction that $C\in(0, \infty)$ and use Assumption \ref{Assumption_2} and the ODE for $\psi$ to see that there exists a constant $A$ such that, for all $x>\kappa$, we have
    \begin{align*}
        \frac{\sigma^{2}(x)}{x}\psi''(x)&=
        -\frac{\mu(x)}{x}\psi'(x)+r\frac{\psi(x)}{x}\\
        &\leq r\frac{\psi(x)}{x}-\frac{(r + c)x + A}{x}\psi'(x)
        \overset{x\rightarrow\infty}{\longrightarrow}
        rC-(r+c)C=-cC<0.
    \end{align*}
    We also find that:
    \begin{equation*}
        \infty>C=\lim_{x\rightarrow\infty}\psi'(x) = \psi'(0)  + \int_{0}^{\infty}\psi''(u)du,
    \end{equation*}
    which implies that there exists a sequence $(x_{k})_{k\in\mathbb{N}}$ increasing to infinity with $\lim_{k\rightarrow\infty}\psi''(x_{k})x_{k}=0$.  Using the above together with  \eqref{limit_sigmasq_over_xsq} we find a contradiction:
    \begin{align*}
       0>-cC\geq \lim_{k\rightarrow\infty}
        \frac{\sigma^{2}(x_{k})}{x^{2}_{k}}\psi''(x_{k})x_{k}=0.
    \end{align*}
\end{enumerate}
\end{proof}

\begin{lemma}\label{Cor:concave-psi} 
    Suppose Assumption \ref{Assumption_1} holds. 
    Consider a fixed constant $b\geq 0$. If $x\mapsto\mu(x)-rx$ is non-decreasing on $[0, b)$ and $\mu(0)\geq0$, then the function $\psi$ (defined in the beginning of Section \ref{sec:Naive_Lebesgue_absolutely_continuous_Nash_equilibria}) is concave on $[0, b)$.
\end{lemma}
\begin{proof}
Using the definition of $\psi$ and $\mu(0) \geq 0$ we find $\displaystyle \psi''(0)=-\frac{2}{\sigma^2(0)} \mu(0)\leq0$. Assume towards a contradiction that $\psi$ is not concave on $[0,b)$. This implies that we should have an $x\in [0,b)$ with $\psi''(x)=0$ and $\psi''>0$ on $(x, x+\varepsilon]$ for some small $\varepsilon>0$ with  $x+\varepsilon \in (x,b)$. Hence, using also the non-decreasingness condition in the statement of the lemma, we find the following:
    \begin{align*}
        \frac{\sigma^2(x+\varepsilon)}{2}\psi''(x+\varepsilon)  
        &= r\psi(x+\varepsilon)- 
        \mu(x+\varepsilon)\psi'(x+\varepsilon), \text{ and }\\
        -r\psi(x) + \mu(x)\psi'(x) &= -\frac{\sigma^2(x)}{2}\psi''(x) = 0\\
        \implies0<\frac{\sigma^2(x+\varepsilon)}{2}\psi''(x+\varepsilon)&=
        r\psi(x+\varepsilon)-r\psi(x)-\mu(x+\varepsilon)\psi'(x+\varepsilon)+\mu(x)\psi'(x)\\
        &=r\int_{x}^{x+\varepsilon}\psi'(u)du-
        \mu(x+\varepsilon)\psi'(x+\varepsilon)+\mu(x)\psi'(x)\\
        &\leq r\varepsilon\psi'(x+\varepsilon)-
        \mu(x+\varepsilon)\psi'(x+\varepsilon)+\mu(x)\psi'(x)\\
        &\leq -\mu(x)\psi'(x+\varepsilon)+\mu(x)\psi'(x)\leq0.
    \end{align*}
Thus, we have arrived at a contradiction.
\end{proof}

\section{Proof of Theorem \ref{Example-with-jumps}}\label{thmproofappendix}
It is immediately verified that the strategies $D^1$ and $D^2$ are admissible (i.e., belonging to $\mathbb{L}$). As in Section \ref{subsec:proof_g_verification}, it suffices to show, for any given $X_{0-}=x>0$, that one of the players, say player $1$, does not gain anything by deviating from $(D^1,D^2)$ (i.e., $J_{1}(x;D, D^2) \leq V_{\ell}(x)$ for any $D \in \mathbb{L}$) and that the expected reward corresponding the strategy tuple $(D^1, D^2)$ satisfies $J_1\left(x;D^1, D^2\right)=V_{\ell}(x)$, where $V_{\ell}$ is given by \eqref{V_b:case-study}.
    
We prove these claims in several steps:
\begin{enumerate}[(I)]
    \item Here we consider $x\in(0, \ell]$ and show that  $J_{1}(x;D^1, D^2)=V_{\ell}(x)$. The existence of a unique strong solution of the SDE corresponding to the candidate equilibrium strategy:
    \begin{equation*}
        dX_{t} = \mu(X_{t})dt + \sigma(X_{t})dW_{t}-
        dL_{t}^{\ell}(X),\enskip X_{0-}=x,\enskip 0\leq t\leq\tau,
    \end{equation*}
    holds by arguments analogous to arguments in Section \ref{subsec:proof_g_verification}. This SDE solution, i.e., the controlled process in equilibrium, is reflected at (the upper barrier) $\ell$; to see this use e.g., \cite[Theorem 3.3]{bass2005one} and \cite[Theorem 2.2]{blei2013note}. Using this, \eqref{V_b:case-study} and the notation $\tau_{n}:=\inf\{t\geq0:X_{t}\geq n\}$, we obtain:
    \begin{align*}
        \E_x\left[e^{-r(\tau\land\tau_n\land n)}V_{\ell}(X_{\tau\land\tau_n\land n})\right]&\leq
        \limsup_{n\rightarrow\infty}
        \E_x\left[e^{-r(\tau_n\land n)}V_{\ell}(\ell)\1_{\{\tau=\infty\}}\right]+
        \E_x\left[e^{-r\tau}V_{\ell}(X_{\tau})\1_{\{\tau<\infty\}}\right]=0.
    \end{align*}
    Note that the above holds since $X_\tau\1_{\{\tau<\infty\}}=0$ and $V_{\ell}(0)=0$ and $\tau_{n}=\infty$ a.s. for large $n$ (recall that we have reflection at $\ell$). Using the definition of $g(\ell)$ and $V_\ell$ we find:
    \begin{equation}
        \frac{V'_{\ell}(\ell+) - V'_{\ell}(\ell-)}{2} - 
        \frac{V'_{\ell}(\ell+)+V'_{\ell}(\ell-)}{2} =
        -\frac{g(\ell)+g(\ell-)}{4},
    \end{equation}
    (which we remark is analogous to the observation \eqref{eq:proof-skew-point-calculation2}). Using the observations above we can now prove $J_{1}(x;D^1, D^2)=V_{\ell}(x)$ for $x\in(0, \ell]$ using arguments based on the It\^{o} formula analogous to those in the proof of Theorem \ref{Theorem_verification_g_NE_2} (cf. the arguments between \eqref{ito-for-other-proof} and \eqref{proofhelp}, with $\lambda^{\ast}_{b}\equiv0$ and $l_{j}=\ell, c_j=\frac{1}{2}$).        \item Here we consider $x\in(0, \ell]$ and show that  $J_{1}(x;D, D^2) \leq V_{\ell}(x)$ for any $D \in \mathbb{L}$.
    The SDE associated with a strategy pair $(D,D^2)$ is:
    \begin{equation*}
        dX_{t} = \mu(X_{t})dt + \sigma(X_{t})dW_{t}-
        \frac{1}{2}dL_{t}^{\ell}(X) - \Delta D_t^2 - dD_t,\enskip X_{0-}=x , 
        \enskip 0 \leq t \leq \tau.
    \end{equation*}
    It suffices to consider the case that this SDE has a unique strong solution (as in the proof of Theorem \ref{Theorem_verification_g_NE_2}). The It\^{o} formula yields:
    \begin{align*}
        e^{-r(\tau_{n}\land\tau\land n)}V_{\ell}(X_{(\tau_{n}\land\tau\land n)-})
        &=V_{\ell}(x) +\int_{0}^{\tau\land\tau_{n}\land n}e^{-rs}
        V'_{\ell}(X_{s})\sigma(X_{s})\1_{\{X_s \neq \ell\}}dW_{s}\\
        &+\int_{0}^{\tau\land\tau_{n}\land n}e^{-rs}
        \left(\left(A_X-r\right)V_{\ell}(X_s)\right)\1_{\{X_s \neq \ell\}}ds\\
        &+\int_{0}^{\tau\land\tau_{n}\land n}e^{-rs}
        \frac{V'_{\ell}(\ell+)-V'_{\ell}(\ell-)}{2}dL^{\ell}_{s}(X)\\
        &-\int_{0}^{\tau\land\tau_{n}\land n}e^{-rs}
        \frac{V'_{\ell}(\ell+)+V'_{\ell}(\ell-)}{4}dL_{s}^{\ell}(X)\\
        &-\int_{0}^{\tau\land\tau_{n}\land n}e^{-rs}
        \frac{V'_{\ell}(X_{s-}+)+V'_{\ell}(X_{s-}-)}{2}dD_s^c\\
        &-\sum_{0\leq s<\tau_n\land\tau\land n}e^{-rs}
        \left(V_{\ell}(X_{s-})-V_{\ell}\left(X_{s-} - \Delta D_s- \Delta D_s^2\right)\right).
    \end{align*}
    Note that the process can only exceed $\ell$ at time $0-$ (cf. the strategy of the non-deviating player, i.e., $D^2$, which dictates jumping above $\ell$) and that, by definition of $V_{\ell}$, we have:
    \begin{equation*}
        \left(A_X-r\right)V_{\ell}(y) = 0, \enskip y \in (0,\ell),
    \end{equation*}
    so that the Lebesgue integral in It\^{o} formula vanishes. Note that Lemma \ref{Cor:concave-psi} and \eqref{eq:g-case-study} yields:
    \begin{equation}\label{help-stuff}
        V_{\ell}'(y)\geq g(y), \enskip y\in[0, \ell). 
    \end{equation}
    
    Using these observations as well as $V_{\ell}'(\ell+)=0$ we find:
    \begin{align*}
        V_{\ell}(x)
        &=\E_x\left[e^{-r(\tau_{n}\land\tau\land n)}
        V_{\ell}(X_{(\tau_{n}\land\tau\land n)-})\right]\\
        &+\E_x\left[
        \int_{0}^{\tau\land\tau_{n}\land n}e^{-rs}\1_{\{X_{s-}=X_{s}=\ell\}}
        \left(\frac{V'_{\ell}(X_{s-}+)+V'_{\ell}(X_{s-}-)}{2}dD_s^c +
        \frac{3V'_{\ell}(\ell-)}{4}dL^{\ell}_{s}(X)\right)\right]\\
        &+\E_x\left[
        \int_{0}^{\tau\land\tau_{n}\land n}e^{-rs}
        \frac{V'_{\ell}(X_{s-}+)+V'_{\ell}(X_{s-}-)}{2}(1 - \1_{\{X_{s-}=X_{s}=\ell\}})dD_s^c\right]\\
        &+\E_x\left[
        \sum_{0\leq s<\tau_n\land\tau\land n}e^{-rs}
        \left(V_{\ell}(X_{s-})-V_{\ell}\left(X_{s-} - \Delta D_s-\Delta D_s^2\right)\right)\right].
    \end{align*}
    If there is a skew point at $\ell$, say with intensity $c'\leq \frac{1}{2}$  included in $D$ (note that we only need to consider this case since a unique strong solution only exists when the intensities sum up to a value that is at most one, cf. \cite{bass2005one,blei2013note} and the strategy of the non-deviating player, i.e., $D^1$, whose skew point intensity at $\ell$ is $\frac{1}{2}$) then it holds that:
    \begin{align*}
        &\1_{\{X_{s-}=X_{s}=\ell\}} \frac{V'_{\ell}(X_{s-}+)+V'_{\ell}(X_{s-}-)}{2}
        dD_{s}^c + 
        \1_{\{X_{s-}=X_{s}=\ell\}} \frac{3V'_{\ell}(\ell-)}{4}dL^{\ell}_{s}(X)\\
        &=\1_{\{X_{s-}=X_{s}=\ell\}}
        \left(\frac{c'}{4} + \frac{3}{8}\right)\frac{g(\ell) + g(\ell-)}{2}dL^{\ell}_{s}(X)\\
        &\geq \1_{\{X_{s-}=X_{s}=\ell\}}c'\frac{g(X_{s-})+g(X_{s-}-)}{2}dL^{\ell}_{s}(X) =
        \1_{\{X_{s-}=X_{s}=\ell\}}\frac{g(X_{s-})+g(X_{s-}-)}{2}dD_{s}^c.
    \end{align*}
    With the above observations we find (using \eqref{help-stuff} repeatedly) that:
    \begin{align*}
        V_{\ell}(x)
        &\geq\E_x\left[
        e^{-r(\tau_{n}\land\tau\land n)}\left(V_{\ell}(X_{(\tau_{n}\land\tau\land n)-})-
        V_{\ell}\left(\left(X_{(\tau_{n}\land\tau\land n)-} - \Delta D_{\tau_{n}\land\tau\land n} - \Delta D^2_{\tau_{n}\land\tau\land n}\right)\lor0
        \right)\right)\right]\\
        &+\E_x\left[\int_{0}^{\tau\land\tau_{n}\land n}e^{-rs}
        \frac{g(X_{s-})+g(X_{s-}-)}{2}\1_{\{X_{s-}=X_{s}=\ell\}}dD_s^c\right]\\
        &+\E_x\left[\int_{0}^{\tau\land\tau_{n}\land n}e^{-rs}
        \frac{g(X_{s-})+g(X_{s-}-)}{2}(1 - \1_{\{X_{s-}=X_{s}=\ell\}})dD_s^c\right]\\
        &+\E_x\left[
        \sum_{0\leq s<\tau_n\land\tau\land n}e^{-rs}
        \left(V_{\ell}(X_{s-})-V_{\ell}\left(\left(X_{s-} - \Delta D_s -
        \Delta D_s^2\right)\lor0\right)\right)\right]\\
        &\geq\E_x\left[e^{-r(\tau_{n}\land\tau\land n)}
        \left(V_{\ell}(X_{(\tau_{n}\land\tau\land n)-})-
        V_{\ell}\left(\left(X_{(\tau_{n}\land\tau\land n)-} - 
        \Delta D_{\tau_{n}\land\tau\land n} - 
        \Delta D^2_{\tau_{n}\land\tau\land n}\right)\lor0\right)
        \right)\right]\\
        &+
        \E_x\left[\int_{0}^{\tau\land\tau_{n}\land n}e^{-rs}
        \frac{g(X_{s-})+g(X_{s-}-)}{2}dD_s^c\right]\\
        &+\E_x\left[
        \sum_{0\leq s<\tau_n\land\tau\land n}e^{-rs}
        \frac{\Delta D_s}{\Delta D_s+\Delta D_s^2}
        \left(V_{\ell}(X_{s-})-V_{\ell}\left(\left(X_{s-} - \Delta D_s-
        \Delta D_s^2\right)\lor0\right)\right)\right]\\
        &\geq\E_x\left[\int_{0}^{\tau\land\tau_{n}\land n}e^{-rs}
        \frac{g(X_{s-}+)+g(X_{s-}-)}{2}dD_s^c\right]\\
        &+\E_x\left[
        \sum_{0\leq s\leq\tau_n\land\tau\land n}e^{-rs}
        \frac{\Delta D_s}{\Delta D_s+\Delta D_s^2}
        \left(V_{\ell}(X_{s-})-V_{\ell}\left(\left(X_{s-} - \Delta D_s - 
        \Delta D_s^2\right)\lor0\right)\right)\right]\\
        &\geq\E_x\left[\int_{0}^{\tau\land\tau_{n}\land n}e^{-rs}
        \frac{g(X_{s-}+)+g(X_{s-}-)}{2}dD_s^c\right]\\
        &+\E_x\left[
        \sum_{0\leq s\leq\tau_n\land\tau\land n}e^{-rs}
        \frac{\Delta D_s}{\Delta D_s+\Delta D_s^2}
        \left(G(X_{s-})-G\left(\left(X_{s-} - \Delta D_s - \Delta D_s^2\right)\lor
        0\right)\right)\right],
    \end{align*}
    from which we conclude that $V_{\ell}(x) \geq J_1\left(x; D, D^2\right)$ as in the proof of Theorem \ref{Theorem_verification_g_NE_2}.
    \item Here we consider $x\in(\ell,\infty)$ and show that  $J_{1}(x;D^1, D^2)=V_{\ell}(x)$.  Using the definitions in \eqref{g-function-payouts2} and \eqref{V_b:case-study} as well the fact that the equilibrium candidate strategies (i.e., $(D^1, D^2)$) imply an immediate jump from $x$ to $\ell$ it is easy to see that:
    \begin{equation*}
        J_{1}\left(x  ; D^1, D^2\right) = 
        J_{1}\left(\ell; D^1, D^2\right) + \frac{1}{2}\int_{\ell}^{x}g(u)du = 
        V_{\ell}(\ell) + 0 = V_{\ell}(x), \enskip x > \ell,
    \end{equation*}
    where the second equality is a direct result of (I).
    \item Here we consider $x\in(\ell,\infty)$ and show that $J_{1}(x;D, D^2)\leq V_{\ell}(x)$ for any $D \in \mathbb{L}$. Consider an arbitrary deviation strategy $D\in\mathbb{L}$. Note that the initial jumps given the pair $(D,D^2)$, i.e., $\Delta D_{0}$ and $\Delta D^{2}_{0}$ are deterministic  (cf. Section \ref{sec:Admissibility}). It follows from the definition in \eqref{g-function-payouts}--\eqref{g-function-payouts2} that:
    \begin{equation}\label{eq:proof_case_study_helper}
        J_1(x, D, D^2) = J_1\left(\left(x - \Delta D_{0} - \Delta D^{2}_{0}\right)\lor0, D, D^2\right) + 
        \frac{\Delta D_0}{\Delta D_0 + \Delta D^{2}_{0}}\int_{\left(x - \Delta D_{0} - \Delta D^{2}_{0}\right)\lor0}^{x}g(u)du.
    \end{equation}
    Observe that $\left(x - \Delta D_{0} - \Delta D^{2}_{0}\right)\lor0\leq \ell$; to see this, observe that the strategy of player 2, i.e., the proposed equilibrium strategy $D^2$, implies an immediate jump of at least the size $x-\ell$, for any deviation strategy of player $1$, cf. Section \ref{sec:Admissibility}. This yields (using also (II)) that:
    \begin{equation*}
        J_1\left(\left(x - \Delta D_{0} - \Delta D^{2}_{0}\right)\lor0, D, D^2\right)\leq V_{\ell}\left(\left(x - \Delta D_{0} - \Delta D^{2}_{0}\right)\lor0\right).
    \end{equation*}
    Using the above and the fact that $g(\ell-)\leq\frac{g(\ell)+g(\ell-)}{4}$ (see \eqref{eq:g-case-study}) we obtain:
    \begin{align*}
        J_{1}\left(x, D, D^2\right)
        &\leq V_{\ell}\left(\left(x - \Delta D_{0} - \Delta D^{2}_{0}\right)\lor0\right) + 
        \frac{\Delta D_{0}}{\Delta D_{0} + \Delta D^{2}_{0}}
        \int_{\left(x - \Delta D_{0} - \Delta D^{2}_{0}\right)\lor0}^{x}g(u)du\\            
        &=V_{\ell}\left(\left(x - \Delta D_{0} - \Delta D^{2}_{0}\right)\lor0\right) + 
        g(\ell-)\frac{\Delta D_{0}}{\Delta D_{0} + \Delta D^{2}_{0}}
        \left(\ell - \left(\left(x - \Delta D_{0} - \Delta D^{2}_{0}\right)
        \lor0\right)\right)\\
        &\leq V_{\ell}\left(\left(x - \Delta D_{0} - \Delta D^{2}_{0}\right)\lor0\right) + 
        \frac{g(\ell) + g(\ell-)}{4}\frac{\Delta D_{0}}{\Delta D_{0} + \Delta D^{2}_{0}}
        \left(\ell - \left(\left(x - \Delta D_{0} - \Delta D^{2}_{0}\right)
        \lor0\right)\right)\\
        &:= f\left(\Delta D_{0}\right).
     \end{align*}
     To prove the desired claim that $J_{1}(x;D,D^2)\leq V_{\ell}(x)$ in the present case it thus suffices to show that $f(\Delta D_{0})\leq V_{\ell}(x)$, for each $\Delta D_{0} \geq 0$. We shall do that in two cases.
     \begin{itemize}
        \item The first case is $\left(x-\Delta D^{2}_{0} - \Delta D_{0}\right)\lor0=0$. This implies that $\displaystyle f(\Delta D_0)=\ell \frac{g(\ell) + g(\ell-)}{4}\frac{\Delta D_{0}}{\Delta D_{0} + \Delta D^{2}_{0}}$. Hence, it suffices to note that: 
        \begin{equation*}
            \ell \frac{g(\ell) + g(\ell-)}{4}
            \frac{\Delta D_{0}}{\Delta D_{0} + \Delta D^{2}_{0}}\leq 
            \ell \frac{g(\ell) + g(\ell-)}{4}\leq
            \frac{\psi(\ell)}{\psi'(\ell)}\frac{g(\ell) + g(\ell-)}{4} = V_{\ell}(x),
        \end{equation*}
        where the first inequality holds by $\frac{\Delta D_{0}}{\Delta D_{0} + \Delta D^{2}_{0}}\leq1$, the second inequality holds by concavity of $\psi$ on $[0, \ell)$ (cf. Lemma \ref{Cor:concave-psi}), and the equality holds by definition \eqref{V_b:case-study}.
        \item The remaining case is $\left(x-\Delta D^{2}_{0} - \Delta D_{0}\right)\lor0>0$. Given the present case, this yields $x - \Delta D_{0} - \Delta D^{2}_{0}\in(0, \ell]$ and we obtain: 
        \begin{equation*}
            \Delta D_{0}\in\left[x - \ell - \Delta D^{2}_{0}, x - \Delta D^{2}_{0}\right).
        \end{equation*}
The inequality $f(\Delta D_{0})\leq V_{\ell}(x)$ is equivalent to the inequality in: 
        \begin{align*}
            f(\Delta D_{0})&=\frac{\psi\left(x - \Delta D_{0} - \Delta D^{2}_{0}\right)}{\psi'(\ell)}\frac{g(\ell) + g(\ell-)}{4} +
            \frac{g(\ell) + g(\ell-)}{4}\frac{\Delta D_{0}}{\Delta D_{0} + \Delta D^{2}_{0}}
            \left(\Delta D_{0} + \Delta D^{2}_{0} + \ell - x\right)\\
            &\leq\frac{g(\ell) + g(\ell-)}{4}
            \frac{\psi(\ell)}{\psi'(\ell)} = V_{\ell}(x),
        \end{align*}
        (cf. also \eqref{V_b:case-study} and the definition $f(\Delta D_{0})$). After some simple manipulations (multiplying both sides of the inequality by $\frac{1}{\frac{g(\ell) + g(\ell-)}{4}}\left(\Delta D_{0} + \Delta D^{2}_{0}\right)$ and rearranging) we see that the above inequality is equivalent to:
        \begin{equation*} 
            h(\Delta D_0) \leq 0,
        \end{equation*}
        where we define:
        \begin{equation*}
            h(y):=y\left(y+\Delta D^{2}_{0} + \ell - x\right)-
            \frac{y+\Delta D^{2}_{0}}{\psi'(\ell)}\left(\psi(\ell)-
            \psi\left(x - y- \Delta D^{2}_{0}\right)\right),\enskip 
            y \in \left[x - \ell - \Delta D^{2}_{0}, x - \Delta D^{2}_{0}\right).
        \end{equation*}
        Hence, it suffices to show that $h\leq0$ on this interval. To this end note that:
        \begin{align*}
            h\left(x - \ell - \Delta D^{2}_{0}\right)&=0\\
            h'(y)&=2y+\Delta D^{2}_{0}+\ell - x \\
            &-\frac{1}{\psi'(\ell)}\left(\psi(\ell)-
            \psi\left(x - y - \Delta D^{2}_{0}\right)+
            \left(y + \Delta D^{2}_{0}\right)
            \psi'\left(x - y - \Delta D^{2}_{0}\right)\right)\\
            h''(y)&=2 - \frac{2}{\psi'(\ell)}
            \psi'\left(x - y- \Delta D^{2}_{0}\right) + 
            \frac{y + \Delta D^{2}_{0}}{\psi'(\ell)}
            \psi''\left(x - y - \Delta D^{2}_{0}\right).
        \end{align*}
        This yields: 
        \begin{align*}
            h'\left(x - \ell - \Delta D^{2}_{0}\right) &= 
            x - \ell - \Delta D^{2}_{0} - (x-\ell) < 0,
        \end{align*}
        and using concavity of $\psi$ on $[0, \ell)$ we also obtain $h''\leq0$ in the relevant interval $\left[x - \ell - \Delta D^{2}_{0}, x - \Delta D^{2}_{0}\right)$ (cf.  $\psi'\left(x - y- \Delta D^{2}_{0}\right)\geq\psi'\left(\ell\right)$). These observations imply that $h \leq 0$ on the relevant interval, which concludes the proof.
    \end{itemize}
\end{enumerate}
\hfill \qedsymbol{} 